\documentclass[11pt]{amsart}
\usepackage[utf8]{inputenc}
\usepackage{amssymb,amsmath}
\usepackage{graphicx}
\usepackage{color}
\usepackage{tikz}
\usepackage{pgfplots}
\usepackage{enumitem}
\usepackage{algorithm}
\usepackage{setspace}
\usepackage{algpseudocode}
\usepackage{mathtools}
\usepackage[hidelinks]{hyperref}
\usepackage{xfrac}

\usepackage{accents}
\usepackage{multicol} 
\usepackage{soul}
\usepackage{subcaption}
\newtheorem{theorem}{Theorem}

\theoremstyle{definition}

\newtheorem{example}[theorem]{Example}

\theoremstyle{lemma}
\newtheorem{lemma}[theorem]{Lemma}

\theoremstyle{remark}
\newtheorem{remark}[theorem]{Remark}

\newtheorem{assumption}[theorem]{Assumption}
\numberwithin{theorem}{section}
\numberwithin{equation}{section}
\numberwithin{table}{section}
\numberwithin{figure}{section}
\usepackage{xspace}
\newcommand{\CN}{Crank--Nicolson\xspace}

\def\Vo{\V_{\ker}}

\def\Vc{\V_\text{c}}

\def\cHo{\cH_{\ker}}
\def\calHker{\cHo}
\def\calAo{\calA_{\ker}}
\def\calDo{\calD_{\ker}}

\def\uo{u_\text{ker}}
\def\uodot{\dot u_\text{ker}}
\def\uoddot{\ddot u_\text{ker}}
\def\uc{u_\text{c}}
\def\ucdot{\dot u_\text{c}}
\def\ucddot{\ddot u_\text{c}}

\def\uex{\widehat{u}}
\def\uexo{\uex_\text{ker}}
\def\uexc{\uex_\text{c}}

\def\w{w}
\def\wo{w_\text{ker}}
\def\wc{w_\text{c}}

\def\wex{\widehat{w}}
\def\wexo{\wex_\text{ker}}
\def\wexc{\w_\text{c}}

\def\wodot{\dot w_\text{ker}}

\def\f{f}

\def\fex{\widehat{f}}

\def\testker{\varphi_\text{ker}}
\def\piker{\pi_\text{ker}}

\newcommand{\Rpm}{\calR_{\pm}}
\newcommand{\Rp}{\calR_{+}}
\newcommand{\Rm}{\calR_{-}}
\newcommand{\RpinvRm}{\calR}

\newcommand{\errunCN}[1]{e_u^{#1}}
\newcommand{\errwnCN}[1]{e_w^{#1}}

\newcommand{\errunGau}[1]{e^{#1}}

\newcommand{\ip}[1]{(#1)}
\newcommand{\dual}[1]{\langle #1\rangle}
\newcommand{\norm}[1]{\| #1\|}

\def\calDo{\calD_{\ker}}
\def\Omo{\Omega_{\ker}}

\def\calAh{ \calA^{\sfrac 12} } 
\def\calAoh{ \calAo^{\sfrac 12} } 

\def\R{\mathbb{R}}

\def\calA{\mathcal{A}}
 
\def\calB{\mathcal{B}}

\def\calD{\mathcal{D}}

\def\cH{\mathcal{H}}
\def\calH{\mathcal{H}}

\def\calK{\mathcal{K}}
\def\calL{\mathcal{L}}

\def\O{\mathcal{O}}
\def\calO{\O}

\def\Q{\mathcal{Q}}
\def\calQ{\mathcal{Q}}
\def\calR{\mathcal{R}}

\def\V{\mathcal{V}}
\def\calV{\V}

\def\calZ{\mathcal{Z}}

\newcommand{\damp}{\calD}
\newcommand{\dampo}{\calDo}

\newcommand{\damppar}{c_{\damp}}

\definecolor{blau}{RGB}{0, 51, 255}
\definecolor{hellblau}{RGB}{153, 204, 255}
\definecolor{hellrot}{RGB}{255, 0, 0}
\definecolor{firebrick}{RGB}{176, 34, 34} 
\definecolor{deep_pink}{RGB}{255, 20, 147} 
\definecolor{sky_blue}{RGB}{74, 112, 139}
\definecolor{slate_blue}{RGB}{71, 60, 139}
\definecolor{chartreuse}{RGB}{118, 238, 0}
\definecolor{chartreuseL}{RGB}{228, 255, 150}
\definecolor{light_blue}{RGB}{178, 223, 238}
\definecolor{dodge_blue}{RGB}{17, 78, 138}
\definecolor{code_backg}{RGB}{238, 216, 174}
\definecolor{myBlue1}{RGB}{101,149,239}  
\definecolor{myBlue2}{RGB}{113,104,238} 
\definecolor{myBlue3}{RGB}{30,144,255} 
\definecolor{myGreen1}{RGB}{154,204,50} 
\definecolor{myGreen2}{RGB}{69,169,0} 
\definecolor{myGreen3}{RGB}{154,205,50} 
\definecolor{myGreen4}{RGB}{105,139,34} 
\definecolor{myRed1}{RGB}{210,105,30} 
\definecolor{myRed2}{RGB}{165,42,42} 
\definecolor{myRed3}{RGB}{139,26,26} 
\definecolor{myLGray}{RGB}{225,225,225} 
\definecolor{mycolor1}{rgb}{0.00000,0.44700,0.74100}%
\definecolor{mycolor2}{rgb}{0.85000,0.32500,0.09800}%
\definecolor{mycolor3}{rgb}{0.92900,0.69400,0.12500}%
\definecolor{mycolor4}{rgb}{0.49400,0.18400,0.55600}%
\definecolor{mycolor5}{rgb}{0.46600,0.67400,0.18800}%
\definecolor{mycolor6}{rgb}{0.30100,0.74500,0.93300}%
\definecolor{mycolor7}{rgb}{0.63500,0.07800,0.18400}%

\DeclareMathOperator{\id}{id}

\DeclareMathOperator{\image}{im}

\DeclareMathOperator{\trace}{trace}

\DeclareMathOperator{\sinc}{sinc}
\renewcommand{\div}{\operatorname{div}}

\newcommand{\dx}{\ensuremath{\, \mathrm{d}x }}
\newcommand{\ds}{\ensuremath{\, \mathrm{d}s }}
\newcommand{\dt}{\ensuremath{\, \mathrm{d}t }}
\newcommand{\deta}{\ensuremath{\, \mathrm{d}\eta }}
\newcommand{\hook}{\ensuremath{\hookrightarrow}}

\newcommand{\tn}[1]{t_{#1}}

\allowdisplaybreaks
\begin{document}
\title[Gautschi-type and IMEX integrators for constrained wave equations]{Gautschi-type and implicit--explicit integrators\\ for constrained wave equations} 
\author[]{R.~Altmann$^\dagger$, B.~D\"orich$^*$, C.~Zimmer$^{\ddagger}$}
\address{${}^{\dagger}$ Institute of Analysis and Numerics, Otto von Guericke University Magdeburg, Universit\"atsplatz 2, 39106 Magdeburg, Germany}
\address{${}^{*}$ 
Institute for Applied and Numerical Mathematics, Karlsruhe Institute of Technology, Englerstraße 2, 76149 Karlsruhe, Germany}
\address{${}^{\ddagger}$ Department of Mathematics, University of Augsburg, Universit\"atsstr.~12a, 86159 Augsburg, Germany}
\email{robert.altmann@ovgu.de, benjamin.doerich@kit.edu, dr.christoph.zimmer@gmail.com}
%
%
\date{\today}
\begin{abstract}
This paper deals with the construction and analysis of two integrators for (semi-linear) second-order partial differential--algebraic equations of semi-explicit type. More precisely, we consider an implicit--explicit \CN scheme as well as an exponential integrator of Gautschi type. For this, well-known wave integrators for unconstrained systems are combined with techniques known from the field of differential--algebraic equations. This results in efficient time stepping schemes that are provable of second order. Moreover, we discuss the practical implementation of the Gautschi-type method, which involves the solution of certain saddle point problems. The theoretical results are verified by a numerical experiment for the wave equation with kinetic boundary conditions. 
\end{abstract}
%
%
\maketitle
%
%
\section{Introduction}

Partial differential equations of second order are widely used to model wave-type phenomena. They often appear in a semi-linear manner with an unbounded linear part and a moderate nonlinearity. In this paper, we consider wave-type equations 
\begin{equation} \label{eq:wave_intro}
	\ddot u(t) 
	+ \damp \dot{u}(t)
	+ \calA u(t)
	= f(t, u)
\end{equation}
with initial conditions~$u(0) = u^0$, $\dot u(0) = w^0$, which underlie an additional (linear) constraint of the form~$\calB u(t) = g(t)$. Typically, $\calA$ and $\damp$ are differential operators; see Examples~\ref{exa:damped_wave} and \ref{exa:wave_constraint} below.
Such systems can be considered as differential--algebraic equations (DAEs) in Banach spaces, also called {\em partial differential--algebraic equations} (PDAEs); see~\cite{LamMT13,Alt15} for an introduction. Second-order PDAEs of wave type appear, e.g., in the field of elastic multibody dynamics~\cite{Sim98,Sim13}, in problems on wave propagation with non-standard boundary conditions~\cite{Alt23}, or, more generally, if two second-order systems are coupled.  

For the time discretization of \eqref{eq:wave_intro} (without a constraint) the two main classes of methods are given by explicit and implicit schemes. 
The advantage of an explicit time-stepping lies in the easy implementation and possibility for parallelization. However, when combined with a spatial discretization small time steps have to be chosen. On the other hand, implicit methods 
are unconditionally stable and thus allow for large time steps. Here, however, a nonlinear system of equations has to be solved in each time step. Thus, it is desirable to use a scheme that takes the best from both methods and obtain a scheme that is unconditionally stable in the linear part, and treats the nonlinear term explicitly. In this work, we consider two types of methods, namely an implicit--explicit \CN scheme (IMEX) as well as an exponential integrator of Gautschi type.

The idea of IMEX schemes is to split the equation in a linear unbounded part that is treated implicitly and a (possibly nonlinear) part treated explicitly, under the condition that the scheme is stable under a step-size restriction that only depends on the Lipschitz condition of the explicit part and preserves the order of convergence of the underlying schemes.
For IMEX schemes, there is a well-developed theory, covering IMEX multistep schemes~\cite{AscRW95,FraHV96,AkrCM99,HunR07} as well as IMEX Runge--Kutta methods~\cite{AscRS97,Bos07}. More recently, an IMEX scheme was proposed that exploits the structure of the second-order system by a combination of the leapfrog method with the \CN method~\cite{HocL21}. This scheme will be the basis of our IMEX method for the constrained problem. Further, there is the so-called \CN--leapfrog scheme proposed in~\cite{LayT12,LayLR16}, which uses the leapfrog method for the first-order system, and is stable only when the explicitly treated part is skew-symmetric.

In the undamped case $\damp = 0$, also Gautschi-type integrators provide a powerful tool for the time integration of wave-type equations. These methods follow the methodology of exponential integrators in the parabolic setting~\cite{HocO10}. Hence, the idea is to integrate the linear (stiff) part in an exact manner, which allows large time steps that are not restricted by the frequencies of the differential operator. 
For second-order wave-type systems (without a constraint), Gautschi-type integrators were developed in~\cite{Gau61,HocL99,Gri02}. For this, one considers the first-order formulation and applies the variation-of-constants formula. The combination of the resulting formula for times $t+\tau$ and $t-\tau$ and an approximation of the integral term by the simplest left/right rectangle rule finally yields an efficient two-step time integration scheme. In the same way as exponential integrators rely on an efficient computation of the matrix exponential applied to a vector~\cite{MolV78}, Gautschi-type integrators are based on the cosine of a matrix (applied to a vector). Also this can be realized efficiently~\cite{HigA10}, e.g., by using Krylov subspace methods~\cite{EieE06,CorKN24}. 

For general constrained systems, one often considers algebraically stable Runge--Kutta methods~\cite{HaiLR89,KunM06,AltZ18}. These methods, however, do not take advantage of the given saddle point structure. Related approaches for systems of first order include splitting methods~\cite{AltO17}, (discontinuous) Galerkin methods~\cite{VouR18,AltH21}, as well as exponential integrators~\cite{AltZ20}. 
In this paper, we aim to translate the previously mentioned concepts to wave-type equations of the form~\eqref{eq:wave_intro} under a linear constraint. For the construction of our novel PDAE integrators, we decompose the solution as in~\cite{EmmM13}. This means that we consider the part of the solution that lies in the kernel of the constraint operator, separately. For the IMEX scheme, we use the half-step formulation from~\cite{HocL21} and impose the linear constraints only on the time derivative of the solution such that the solution $u$ satisfies the constraint exactly. This leads to a method that requires (after the spatial discretization) the solution of two (linear) saddle point problems in each time step, one including the mass and one the stiffness matrix. 

For the Gautschi-type integrators, we again consider only the part of the solution in the kernel. Together with techniques developed for constrained parabolic systems in~\cite{AltZ20}, this then leads to the desired explicit integration scheme. To prove (mesh-independent) convergence of this novel integrator, the theory of~\cite{HocL99,Gri02} for the unconstrained case is combined with the already mentioned PDAE techniques. 
In view of the efficiency of the practical implementation, the critical part is the computation of $\cos(\tau\Omega_{\ker})v$ for some vector $v$. Here, $\Omega_{\ker}$ denotes the root of the differential operator~$\calA$ (or its discrete counterpart) restricted to the kernel of the constraint operator~$\calB$. For this, we propose an Arnoldi method that does not require the computation of~$\Omega_{\ker}$; cf.~\cite{Gri05b} for the unconstrained case. 
Moreover, by particularly constructed saddle point problems, the computations can be restricted to computations on the kernel of the constraint operator. 
At this point, we would like to emphasize that Gautschi-type integrators are tailored to problems with vanishing damping, i.e., $\damp = 0$. In order to include damping in an exponential manner, one could also use general exponential methods of Runge--Kutta or multistep type; cf.~the review~\cite{HocO10}. This, however, is beyond the scope of this work. 
\medskip

The paper is organized as follows. In Section~\ref{sec:PDAE} we introduce the system equations in detail and discuss its unique solvability. Afterwards, we turn to the numerical treatment of theses systems. First, an implicit--explicit version of the \CN scheme is introduced and analyzed in Section~\ref{sec:CN}. Here, a detailed inspection of all the error contributions allows us to prove an error equation, which can be resolved by the techniques developed for the unconstrained IMEX scheme. It follows the discussion of Gautschi-type integrators for damping-free systems in Section~\ref{sec:Gautschi}. Besides proving second-order accuracy of the method, we also provide comments on the practical implementation of the algorithm. Finally, we present a numerical example in Section~\ref{sec:numerics} and conclude in Section~\ref{sec:conclusion}. 
%
%
\section{Wave-type Systems with Linear Constraints}
\label{sec:PDAE}
This section is devoted to the introduction of semi-linear PDAEs of wave type. Hence, we consider systems of the form~\eqref{eq:wave_intro} that underlie an additional (linear) constraint. 
%
%
\subsection{Problem formulation}
\label{sec:PDAE:prelim}
In this subsection, we introduce a class of hyperbolic wave-type systems, restricted by a linear constraint. Allowing (sufficiently smooth) nonlinearities in the low-order terms of the dynamic equation, we end up with the following PDAE: find abstract functions~$u\colon[0,T] \to \V$ and~$\lambda\colon[0,T] \to \Q$ such that   
\begin{subequations}
\label{eq:PDAE}
\begin{alignat}{5}
	\ddot{u}(t) + \damp \dot{u}(t) &\ +\ &\calA u(t)&\ +\ &\calB^* \lambda(t)\, 
	&= f(t, u) &&\qquad\text{in } \V^* \label{eq:PDAE:a},\\
	& &\calB u(t)& & 
	&= g(t) &&\qquad\text{in } \Q^*.
	\label{eq:PDAE:b}
\end{alignat}
\end{subequations}
Therein, the spaces~$\V$ and $\Q$ are assumed to be Hilbert spaces with respective duals~$\V^*$ and~$\Q^*$. Further, the space $\V$ is part of a Gelfand triple~$\V$, $\cH$, $\V^*$, i.e., we assume a continuous (and dense) embedding~$\V \hook \cH$,. This, in turn, implies a corresponding embedding of the duals~\cite[Ch.~23.4]{Zei90a}.
The appearing operators are linear and satisfy~$\calA \in \calL(\V,\V^*)$, $\calB \in \calL(\V,\Q^*)$, and $\calD \in \calL(\cH,\V^*)$. 
Moreover, equation~\eqref{eq:PDAE:a} includes the Lagrange multiplier~$\lambda$, which is used to enforce the constraint~\eqref{eq:PDAE:b} following the Lagrangian approach. 

For second-order systems in the weak form, one typically seeks a solution with the first and second derivative taking values in~$\cH$ and~$\V^*$, respectively. Hence, initial data is expected to be of the form~$u(0) = u^0 \in \V$, $\dot u(0) = w^0 \in \cH$. Note, however, that the constraint~\eqref{eq:PDAE:b} typically implies certain consistency conditions on the initial data; cf.~\cite[Rem.~7.4]{Alt15}. 
At this point, it is reasonable to assume~$u^0, w^0 \in\V$ with~$\calB u^0 = g(0)$ and~$\calB w^0 = \dot g(0)$. For the linear operators we make the following assumptions. 
\begin{assumption}[Constraint operator $\calB$]
\label{assB}
The operator~$\calB\colon \V \to \Q^*$ is linear, continuous, and satisfies an inf--sup condition, i.e., there exists a constant~$\beta>0$ such that 
\[
	\adjustlimits \inf_{q\in\Q\setminus\{ 0 \} }\sup_{v\in\V\setminus\{0\}}
	\frac{\langle \calB v, q\rangle}{\Vert v\Vert_\V \Vert q\Vert_\Q}
	\ge \beta. 
\]
\end{assumption}
\begin{assumption}[Differential operator $\calA$]
\label{assA}
The operator~$\calA\colon \V\to\V^*$ is linear, continuous, and elliptic on~$\Vo \coloneqq \ker\calB$, i.e., on the kernel of the constraint operator. Furthermore, we assume that we can decompose the operator into~$\calA = \calA_1+\calA_2$, with $\calA_1$ being self-adjoint, i.e., $\langle \calA_1 u, v\rangle = \langle u, \calA_1 v\rangle$, and $\calA_2 \in \calL(\calV,\calH^*)$. In particular, we may assume that~$\calA_1$ is elliptic on $\Vo$ as well; cf.~\cite{AltZ20}.  
\end{assumption}
%
%
\begin{assumption}[Damping operator $\damp$]
\label{assD} 
The operator~$\damp \colon \cH \to \V^*$ is linear and continuous. Moreover, its restriction to functions in $\V$ satisfies $\damp \big|_\V \colon \V \to \cH^*$, is again continuous, and satisfies the monotonicity condition
\begin{equation} 
\label{eq:damp_monotone}
	\ip{\damp u, u}_{\cH} 
	\ge 0 
	\quad 
	\text{ for all }
	\quad
	u \in \V.
\end{equation}
\end{assumption}
\begin{remark}
The previous assumptions can be generalized to the case where $\calA$ only satisfies a G\aa{}rding inequality; see~\cite{AltZ20}, or $\calD$ satisfies a lower bound of the type $\ip{\damp u , u}_{\cH} \geq- \damppar \norm{u}^2_{\cH}$. 
\end{remark}
As discussed in~\cite{AltZ20}, Assumption~\ref{assB} implies the existence of a continuous right-inverse~$\calB^-\colon\Q^*\to\V$. This then leads to the decomposition
\begin{align}
\label{eq:decompV}
	\V = \Vo \oplus \Vc \qquad\text{with}\qquad
	\Vo = \ker\calB,\quad 
	\Vc = \image\calB^-,  
\end{align}
where we denote by $\oplus$ the direct sum of two vector spaces. This means, in particular, that $\Vo \cap \Vc = \{0\}$.
Following~\cite{AltZ18}, we define the complementary space~$\Vc$ as 
\[
	\Vc 
	\coloneqq \big\{ v\in\V\, |\, \calA v \in \Vo^0 \big\}
	= \big\{ v\in\V\, |\, \langle \calA v, w\rangle = 0 \text{ for all } w\in \Vo \big\}.
\]
%
We also introduce the restriction of the differential operator to the kernel of~$\calB$, namely  
\[
	\calAo
	\coloneqq \calA|_{\Vo} \colon \Vo\to\Vo^* \coloneqq  (\Vo)^*. 
\]
Here, we use the fact that functionals in~$\V^*$ define functionals in~$\Vo^*$ simply through the restriction to~$\Vo$. Note that~$\calAo$ is elliptic due to Assumption~\ref{assA}. Hence, $-\calAo$ generates an analytic semigroup on the closure of~$\Vo$ in the $\cH$-norm, i.e., on~$\cHo \coloneqq \operatorname{clos}_\cH \Vo$; cf.~\cite[Ch.~7, Th.~2.7]{Paz83}. 

Similarly, we introduce the restriction 
\[
	\calDo
	\coloneqq \calD|_{\cHo} \colon \cHo\to\Vo^*, 
\]
which satisfies $\calDo \big|_{\Vo} \colon \Vo \to \cHo^*$ according to Assumption~\ref{assD}. 

\begin{remark}
With the space $\cHo$ at hand, the assumptions on the initial data may be relaxed to $w^0\in\calHker+\calB^-\dot{g}(0)$, i.e., $w^0$ can be written as the sum of $\calB^-\dot{g}(0)\in \Vc$ and an element of $\calHker$, which is still consistent with the constraint.
\end{remark}
%
Finally, we summarize the assumptions on the right-hand sides. Within the constraint, we assume a mapping~$g\colon [0,T] \to \Q^*$. Moreover, the nonlinearity satisfies~$f\colon [0,T] \times \cH \to \cH^*$. In addition, we require the following Lipschitz bounds on $f$, depending on the method we apply.

\begin{assumption}[Nonlinearity $f$]
\label{assRHS}
Let $f$ satisfy the classical Carath{\'e}odory condition and one of the following two conditions:

\noindent
(a) There is a constant $L_{\cH}$ such that for all $v_1,v_2 \in \cH$ and $t\geq 0$ it holds
\begin{equation*}
	\|f( t ,v_1)-f(t,v_2)\|_{\calH^*} 
	\le L_{\cH}\, \|v_1-v_2\|_{\calH}  .
\end{equation*}

\noindent
(b) There is a constant $L_{\V}(\rho)$ such that for all $v_1,v_2 \in \V$ with $\norm{v_i}_{\V} \leq \rho$
 and $t\geq 0$ it holds
\begin{equation*}
	\|f(t,v_1)-f(t,v_2)\|_{\calH^*} 
	\le L_{\V}(\rho)\, \|v_1-v_2\|_{\V} .
\end{equation*}
\end{assumption}
\begin{remark}
Within Assumption~\ref{assRHS}, condition (a) implies condition (b).
\end{remark}
\begin{example}	 
\label{exa:damped_wave}
In this first example, we focus on the assumptions in the non-constrained part of the problem. Thus, consider the wave equation with additional convection and damping terms, 
\[
	\ddot u + (\alpha + \beta \cdot \nabla)\, \dot{u} - \Delta u + \bar v\cdot\nabla u
	= f(u)
	\qquad\text{in }\Omega \subseteq\R^3
\]	
with $\alpha,\bar v\in L^\infty(\Omega)$ and $\beta \in H(\div,\Omega) \cap L^\infty(\Omega)$ satisfying $\alpha - \frac12 \div \beta \geq 0$ almost everywhere,
and homogeneous Dirichlet boundary conditions. We set~$\calV=H^1_0(\Omega)$ and $\calH=L^2(\Omega)$. Then the operator $\calA$ is given by 
\[
	\langle \calA u, w\rangle 
	= \int_\Omega \nabla u \cdot \nabla w + \bar v\cdot\nabla u\, w \dx
	\eqqcolon \langle \calA_1 u, w\rangle  + \langle \calA_2 u, w\rangle. 
\]
Here, $\calA_1$ corresponds to the Laplacian, which is elliptic and self-adjoint, whereas~$\calA_2$ models the convection term. Note that~$\calA_2$ is an element of $\calL(\calV,\calH^*)$, since we do not need to differentiate the test function. 
For the damping operator, we consider $\damp \big|_\V = \alpha + \beta \cdot \nabla$. Integration by parts gives 
\begin{equation}
	\ip{\damp u , u}_{\cH} 
	= \int_\Omega  (\alpha - \tfrac12 \div \beta)\, |u|^2 \dx \geq 0 .
\end{equation}
Similarly, integration by parts gives that for $u \in \cH$ and $v \in \V$ it holds that 
\begin{equation}
	 \dual{\damp u , v}_{\V^*, \V} 
	 = \int_\Omega  (\alpha - \div \beta)\, u v - u \beta \cdot \nabla v \dx.
\end{equation}
For $f(z) = \sin(z)$, Assumption~\ref{assRHS}~(a) is satisfied. For $f(z) =|z|^{p-1} z$, $1\leq p \leq 3$, one can verify Assumption~\ref{assRHS}~(b) by Sobolev's embedding.
\end{example}
\begin{example}	 
\label{exa:wave_constraint}
In our second example, we focus on the constraint. For this, we consider
\[
	\ddot u  - \Delta u = f(t)
	\qquad\text{in }\Omega \subseteq\R^3
\]	
subject to $\calB u \coloneqq \trace u = g \in H^{1/2} (\Gamma)$ for $\Gamma \coloneqq \partial \Omega$.
Here, we have $\calV = H^1(\Omega)$ and $\calH = L^2(\Omega)$ as well as $\Q = H^{-1/2}(\Gamma)$. Thus, we obtain 
$\Vo = H^1_0(\Omega)$ and  $\cHo = \calH$.
As a result, the operator $\calAo$ equals the Dirichlet Laplacian and is hence elliptic on $\Vo$.
\end{example}
A more sophisticated example of a wave system with kinetic boundary conditions is postponed to Section~\ref{sec:numerics}, where we present a numerical example.
%
%
\subsection{Existence of solutions}
\label{sec:PDAE:existence}
Before introducing numerical schemes for constrained systems of wave type, we need to discuss the existence of solutions to~\eqref{eq:PDAE}. Throughout this section, we denote the well-known Sobolev-Bochner spaces for a Banach space $X$ by~$L^2(0,T;X)$ and~$H^1(0,T;X)$; cf.~\cite[Ch.~23]{Zei90a}. Moreover, $C([0,T],X)$ denotes the space of continuous functions in $[0,T]$ with values in $X$. 
As a first step, we give a classical existence result for the linear case without constraints. Here, linear means that the right-hand side~$f$ is independent of~$u$. 
\begin{lemma}[Existence of solutions, linear and unconstrained case]
\label{lem:existenceLinUnconstrained}
Let Assumption~\ref{assD} hold and consider $\calA = \calA_1+\calA_2$ with $\calA_1 \in \calL(\calV,\calV^*)$ being elliptic with constant~$\mu_1$ and self-adjoint and $\calA_2 \in \calL(\calV,\calH^*)$ with continuity constant~$C_2$. Further, suppose that $u^0 \in \calV$, $w^0 \in \calH$, and $f \in L^2(0,T;\calH^*)$ are given. Then the linear equation
\[
	\ddot u(t) + \damp \dot{u} (t) + \calA u(t) 
	= f(t)
	\qquad\text{in } \V^* 
\]
with initial conditions $u(0) = u^0$ and $\dot u(0) = w^0$ has a unique solution 
\begin{align*}
	u \in C([0,T],\calV), \qquad 
	\dot{u} \in C([0,T],\calH), \qquad 
	\ddot{u} \in L^2(0,T;\calV^*). 
\end{align*}
This solution satisfies the estimate
\begin{align}
\label{eq:stabEstimate:noConst}
	\|\dot u(t)\|_{\calH}^2 + \|u(t)\|_{\calA_1}^2 
	\leq e^{ \kappa t}\,
	\Big( \sqrt{\|w^0\|_{\calH}^2 + \|u^0\|_{\calA_1}^2} + \int_0^t e^{-\kappa s}\, \|f(s)\|_{\calH^*} \ds \Big)^2
\end{align}
with exponent~$\kappa \coloneqq C_2 / \sqrt{\mu_1}$ and the problem-dependent norm~$\|\bullet\|^2_{\calA_1}\coloneqq\langle \calA_1\, \bullet\,, \bullet\rangle$. 
\end{lemma}
\begin{proof}
By an adaption of~\cite[Cor.~24.3]{Zei90a}, the problem has a unique solution~$u$ within the given spaces. It remains to show the stability estimate~\eqref{eq:stabEstimate:noConst}. For this, note that~$u$ satisfies the system
\[
	\ddot u(t) 
	+
	\damp \dot{u} (t)
	+
	 \calA_1 u(t) 
	= \tilde{f}(t)
	\coloneqq f(t) - \calA_2 u(t) 
\]
with the same initial values and right-hand side~$\tilde f = f - \calA_2 u \in L^2(0,T;\calH^*)$. Along the same lines as in~\cite[Ch.~3, Lem~8.3]{LioM72a}  
and Young's inequality, we conclude that 
\begin{align*}
	\|\dot u(t)\|_{\calH}^2 + \|u(t)\|_{\calA_1}^2 
	&\leq \|w^0\|_{\calH}^2 + \|u^0\|_{\calA_1}^2 + 2\int_0^t \big\langle f - \calA_2 u, \dot{u} \big\rangle \ds \\
	&\leq  \|w^0\|_{\calH}^2 + \|u^0\|_{\calA_1}^2 + \kappa \int_0^t \| \dot u\|_{\calH}^2 + \| u\|_{\calA_1}^2 \ds + 2\int_0^t \|f\|_{\calH^*}\|\dot{u}\|_{\calH} \ds.
\end{align*}
Estimate~\eqref{eq:stabEstimate:noConst} then follows along the lines of~\cite[Th.~4.22]{Zim21}. By the linearity of the problem, this estimate also proves the uniqueness of the solution.
\end{proof}
We return to the constrained case, i.e., to the PDAE~\eqref{eq:PDAE}. 
\begin{lemma}[Existence of solutions, linear PDAE case]
\label{lem:existenceLinPDAE}
Consider Assumptions~\ref{assB}, \ref{assA}, and~\ref{assD} as well as right-hand sides~$f \in L^2(0,T;\calH^*)$ and~$g \in H^2(0,T;\calQ^*)$. Further assume consistent initial data, namely  
\[
	u^0 \in \calB^-g(0) + \Vo, \qquad 
	w^0 \in \calB^-\dot g(0) + \cHo.
\] 
Then there exists a unique solution to~\eqref{eq:PDAE} with $u(0) = u^0$, $\dot u(0) = w^0$, satisfying 
\begin{align*}
	u \in C([0,T],\calV), \qquad 
	\dot{u} \in C([0,T],\calH), \qquad 
	\ddot{u} \in L^2(0,T;\Vo^*), 
\end{align*}
as well as a stability estimate of the form  
\begin{align} 
	&\|\dot u(t) \|_{\calH} 
	+ \| u(t)\|_\calV
	\lesssim \|w^0\|_{\calH} + \|u^0 \|_\calV 
		 + \int_0^t \|f(s)\|_{\calH^*} \ds 
		 + \|g\|_{H^2(0,t;\calQ^*)}.
	 \label{eq:stabEstimate:Constraints}
\end{align}
\end{lemma}
\begin{proof}
Considering the decomposition $u = \uo + \uc$ given by $\V = \Vo \oplus \Vc$, we directly obtain $\uc(t) = \calB^-g(t)\in\Vc$. For $\uo$, we restrict~\eqref{eq:PDAE:a} to test functions in the kernel space~$\Vo$. Due to the definition of~$\Vo$, the Lagrange multiplier vanishes for such test functions and we get  
\[
	\uoddot
	 +
	\damp \uodot
	+
	\calAo \uo
	= \tilde f(t)
	\coloneqq f(t)
	 -
	 \damp \ucdot
		-
	 \ucddot
	= 
	f(t) 
	-
	\damp \calB^- \dot g(t)
	- 
	\calB^- \ddot g(t)
	\quad\text{in } \Vo^*.
\]
The right-hand side satisfies $\tilde f \in L^2(0,T;\calH^*)$. Moreover, the corresponding initial data is given by $\uo(0) = u^0 - \calB^-g(0)\in \Vo$ and $\uodot(0) =  w^0 - \calB^-\dot g(0) \in\cHo$. Since~$\Vo$, $\cHo$, $\Vo^*$ defines a Gelfand triple, we can apply Lemma~\ref{lem:existenceLinUnconstrained}, leading to a unique solution 
\begin{align*}
	\uo \in C([0,T],\Vo), \qquad 
	\uodot \in C([0,T],\cHo), \qquad 
	\uoddot \in L^2(0,T;\Vo^*)
\end{align*}
and the stated stability estimate. 
%
\end{proof}
Finally, we transfer the existence results to the semi-linear case. 

\begin{theorem}[Existence of solutions]
\label{th:existencePDAE}
Consider Assumptions~\ref{assB}, \ref{assA}, \ref{assD}, and \ref{assRHS} as well as a right-hand side~$g \in H^2(0,\widetilde{T};\calQ^*)$. Further, assume consistent initial data 
\[
	u^0 \in \calB^-g(0) + \Vo, \qquad 
	w^0 \in \calB^-\dot g(0) + \cHo.
\] 
Then there exists a	$0 <T \leq \widetilde{T}$ and a unique solution to~\eqref{eq:PDAE} on $[0,T]$ with $u(0) = u^0$, $\dot u(0) = w^0$ satisfying 
\begin{align*}
	u \in C([0,T],\calV), \qquad 
	\dot{u} \in C([0,T],\calH), \qquad 
	\ddot{u} \in L^2(0,T;\Vo^*). 
\end{align*}
Moreover, we have the stability estimate that $\|\dot u(t)\|_{\calH}^2 + \|u(t)\|_{\calV}^2$ stays bounded. 
%
\end{theorem}
\begin{proof}
Again, the decomposition $u = \uo + \uc$ yields that we only have to study the part in the kernel, and we proceed as in Lemma~\ref{lem:existenceLinPDAE}. 
We employ estimate~\eqref{eq:stabEstimate:Constraints} to set up a fixed-point argument with the nonlinearity 
\[
	\widehat f(\,\cdot\,,\uo) 
	\coloneqq 
	f(\,\cdot\,,\uo+\calB^-g)
	-
	\damp \calB^- \dot g(t)
	- 
	\calB^- \ddot g(t) ,
\]
which satisfies for $v_1,v_2 \in \Vo$ the Lipschitz condition 
\[
	\big\| \widehat f(t, v_1) - \widehat f(t, v_2) \big\|_{\cH}
	= \big\| f(t,v_1+\calB^-g(t)) - f(t,v_2+\calB^-g(t)) \big\|_{\cH}
	\le L_\V\, \|v_1-v_2\|_{\V}.
\]
This then yields the claimed well-posedness.
\end{proof}
%
%
%
\begin{remark}

In this paper, we focus on the temporal discretization and only shortly discuss the spatial discretization of the PDAE~\eqref{eq:PDAE} using a (conforming) finite element approximation. Such a discretization yields a symmetric, positive definite mass matrix~$M\in \R^{n,n}$, a stiffness matrix~$A\in \R^{n,n}$ corresponding to the operator~$\calA$, a damping matrix~$D\in \R^{n,n}$ corresponding to the operator~$\calD$, and a constraint matrix~$B\in \R^{m,n}$ with~$m\le n$, which we assume to be of full (row) rank. According to the assumptions on~$\calA$ from Assumption~\ref{assA}, we further assume that the stiffness matrix~$A$ is positive on the kernel of~$B$. The resulting semi-discrete system then reads 
\begin{subequations}
\label{eq:DAE}	
\begin{alignat}{5}
	M\ddot{x}(t) + D\dot{x}(t)&\ +\ & A x(t)&\ +\ & B^T \lambda(t)\, 
	&= f(t,x(t)), \label{eq:DAE:a}\\
	& & B x(t)& & 
	&= g(t), \label{eq:DAE:b}
\end{alignat}
\end{subequations}
where the right-hand sides are now functions $f\colon [0,T] \times \R^n \to \R^n$ and~$g\colon [0,T] \to \R^m$. Moreover, we assume consistent initial values, i.e., $x(0) = x^0$, $\dot x(0) = y^0$ with $Bx^0 = g(0)$ and $B y^0 = \dot g(0)$.  
Due to the semi-explicit structure of semi-discrete system, one can easily see that the stated assumptions imply that the DAE~\eqref{eq:DAE} is of index~$3$; cf.~\cite[Ch.~VII.1]{HaiW96}.
\end{remark}
%
%
\section{An implicit--explicit \CN Scheme}
\label{sec:CN}
This section is devoted to the extension of the IMEX scheme introduced in~\cite{HocL21} to constrained systems. 
For this, we consider a uniform partition of the time interval~$[0,T]$ into~$0=\tn{0}< \tn{1} < \dots < \tn{N} = T$ with constant time step size~$\tau$, i.e., $\tn{n}=n\, \tau$. Furthermore, we denote by~$u^n$ the computed approximation of~$u$ at time $\tn{n}$ and abbreviate function evaluations by~$g^n \coloneqq g(\tn{n})$. Throughout, we use the notation
\begin{equation*}
	a \lesssim b,
\end{equation*}
if there is a constant $C > 0$ independent of the step size $\tau$ such that $a \leq C b$.
%
%
\subsection{Review of the unconstrained case} 
\label{sec:CN:unconstrained}
We first recall the situation for the unconstrained case. For this, we consider a system of the form 
\[
	\ddot u(t) 
	+ \damp \dot u(t) 
	+ \calA u(t)
	= f(t, u)
\]
with initial conditions~$u(0) = u^0$, $\dot u(0) = w^0$. A \CN discretization, where the nonlinear term is treated by the left and right rectangle rule (see~\cite{HocL21} for more details) yields the IMEX scheme
\begin{subequations}
	\label{eq:IMEX:unconstrained:halfstep}
	\begin{align}
		\big(\id +  \tfrac{\tau}{2} \damp + \tfrac{\tau^2}{4} \calA 
		\big)\, w^{n+1/2}
		&= w^n - \tfrac{\tau}{2} \calA u^n  + \tfrac{\tau}{2} f^n, \\ 
		u^{n+1} 
		&= u^n + \tau w^{n+1/2}, \\ 
		w^{n+1}
		&= -w^n + 2w^{n+1/2} + \tfrac{\tau}{2} \big( f^{n+1} - f^n\big),
	\end{align}
\end{subequations}
where $f^n = f(\tn{n},u^n)$. We would like to emphasize that this scheme is {\em explicit} in the nonlinearity such that no nonlinear solver is needed. Moreover, it has been shown to remain second-order accuracy.
\begin{theorem}[Second-order convergence, unconstrained case {\cite[Th.~2.9]{HocL21}}]
\label{th:IMEX:convergence}
Consider initial data $u^0$ in the domain of~$\calA$ and $w^0\in \V$ as well as the unique solution~$u$ satisfying 	
\begin{align*}
	u \in C^3([0,T], \V) \cap C^4([0,T], \cH), \qquad
	f(\,\cdot\,, u(\cdot)) \in C^2([0,T], \cH). 
\end{align*}
Let Assumption~\ref{assRHS}~(b) hold. Then the discrete solution obtained by the IMEX scheme~\eqref{eq:IMEX:unconstrained:halfstep} satisfies the convergence result
\[
	\|	u(\tn{n}) - u^n \|_{\V}
	+
	\|	\dot u(\tn{n}) -w^n \|_{\cH}
	\lesssim
	\tau^2.
\]	
\end{theorem}
%
Let us already mention that in the case of $\damp = 0$, one can derive an equivalent two-step formulation, eliminating $w^n$; see Lemma~\ref{lem:eq_two_step_undamped} below. 
%
\subsection{An implicit--explicit scheme of second order}
\label{sec:CN:scheme}
We now turn to the PDAE system~\eqref{eq:PDAE} with consistent initial data. 
Following the approach of the previous subsection, we propose to compute in each time step the following three substeps. First, compute the pair $(w^{n+1/2}, \lambda^{n+1}) \in \calV\times \calQ$ via 
\begin{subequations} \label{eq:IMEX_damp} 
	\begin{align}
		\big( \id + \tfrac{\tau}{2} \damp + \tfrac{\tau^2}{4} \calA \big)\, w^{n+1/2} + \tfrac{\tau}{2}  \calB^* \lambda^{n+1}
		&= w^n - \tfrac{\tau}{2} \calA u^n  
		+ \tfrac{\tau}{2} f^n, 
		\label{eq:IMEX_damp_a}
		\\ 
		\calB w^{n+1/2} \hspace{1.95cm}
		&= \tfrac{1}{\tau}\, \big( g^{n+1} - g^n \big).
		\label{eq:IMEX_damp_b}
	\end{align}	
	Note that this system has a unique solution, since the operator in brackets on the left-hand side is elliptic on $\Vo$ and $\calB$ is inf--sup stable; see~\cite[Lem~3.1]{AltZ20}. Second, we obtain $u^{n+1} \in \calV$ by
	\begin{align}
		u^{n+1} 
		= u^n + \tau w^{n+1/2}. 
		\label{eq:IMEX_damp_c}
	\end{align}
	Finally, we consider the system 
	\begin{align}
		w^{n+1} + \calB^* \mu^{n+1} 
		&= 2w^{n+1/2} - w^n + \tfrac{\tau}{2}\, \big( f^{n+1} - f^n \big), 
		\label{eq:IMEX_damp_d}
		\\
		\calB w^{n+1} \hspace{1.72cm}
		&= \tfrac{1}{\tau}\, \big( g^{n+3/2} - g^{n+1/2} \big).
		\label{eq:IMEX_damp_e}
	\end{align}
\end{subequations}
Since there is no $\Vo$-elliptic operator in front of $w^{n+1}$ in the first equation, the existence of a unique solution $w^{n+1} \in \cH_{\ker} + \Vc$ is only guaranteed if $\cHo$ is a proper subspace of $\cH$; cf.~\cite[Lem.~9.22]{Zim21}. In this case, the expression $\calB w^{n+1}$ is indeed well-defined, even though $w^{n+1} \not\in \V$. In the following, we assume the unique solvability of~\eqref{eq:IMEX_damp_d}-\eqref{eq:IMEX_damp_e} such that the IMEX scheme~\eqref{eq:IMEX_damp} is well-defined overall. 

In the case where $\cHo$ is \emph{no} proper subspace of $\cH$, we may replace 
\eqref{eq:IMEX_damp_d} and \eqref{eq:IMEX_damp_e} by
\begin{equation}  
	\label{eq:variant_H_Hker_IMEX}
w^{n+1}
	= 2w^{n+1/2} - w^n + \tfrac{\tau}{2}\, \big( f^{n+1} - f^n \big)
\end{equation}
and later decompose $w^{n+1} =  w^{n+1}_{\ker}	+ w^{n+1}_c$, where we choose 
$w^{n+1}_c =  \tfrac{1}{\tau}\, \calB^-	\big( g^{n+3/2} - g^{n+1/2} \big)$. This decomposition, however, is only for theoretical purposes and not needed in the implementation of the method. In addition, we emphasize that the methods are not equivalent in the proper subspace case, but will produce the same approximations $u^n$ in the case $\damp = 0$ and have the same error bounds, as we will explain in detail in Remark~\ref{rem:H_eq_H_ker} below.
\begin{remark}
\label{rem:consistent}
One can easily see by mathematical induction that $u^{n+1}$ is consistent, i.e., 
\[
	\calB u^{n+1}
	= \calB u^{n} + \tau \calB w^{n+1/2}
	= \calB u^{n} + g^{n+1} - g^n
	= g^{n+1},
\]
as long as $u^0$ is consistent. 
\end{remark}
\begin{remark}
\label{rem:saddlepointproblems}	
In the practical implementation of the method, we would first perform a spatial discretization, e.g., by finite elements, leading to matrices $A$, $D$, and $B$ as discrete versions of $\calA$, $\calD$ and $\calB$, respectively. Moreover, we obtain a mass matrix~$M$ that appears in front of the second-order terms. Systems \eqref{eq:IMEX_damp_a}--\eqref{eq:IMEX_damp_b} and \eqref{eq:IMEX_damp_d}--\eqref{eq:IMEX_damp_e} then lead to saddle point problems with matrices 
\[
	\begin{bmatrix}
		M + \tfrac\tau2 D + \tfrac{\tau^2}{4} A & \tfrac\tau2 B^T \\  B 
	\end{bmatrix}
 	\qquad\text{and}\qquad
	\begin{bmatrix}
		M & B^T \\ B 
	\end{bmatrix}.
\]
\end{remark}
\begin{lemma} 
\label{lem:eq_two_step_undamped}
In the damping-free case, i.e., $\damp = 0$, scheme~\eqref{eq:IMEX_damp} is equivalent to the two-step formulation
\begin{subequations} 
\label{eq:IMEX_two_step}
	\begin{align}
		u^{n+1} - 2 u^{n} + u^{n-1}  
		+
		\calB^* \widetilde{\lambda}^{n+1}
		&=
		-\tfrac{\tau^2}{4} \calA \big(u^{n+1} + 2 u^{n} + u^{n-1} \big)
		+
		\tau^2 f^{n} ,
		\\
		\calB u^{n+1} 
		&= g^{n+1}  ,
	\intertext{if the first step is computed via}
		u^{1} +
		\calB^* \widetilde{\lambda}^{1} 
		&=  u^{0} + \tau w^{0} 
		-
		\tfrac{\tau^2}{4} \calA \big( u^{0} + u^{1} \big) 
		+
		\tfrac{\tau^2}{2} f^{0} ,
		\\
		\calB u^{1} 
		&= g^{1} .
	\end{align}
\end{subequations}
\end{lemma}
\begin{proof}
We rewrite \eqref{eq:IMEX_damp_a} as
\begin{align*}
	w^{n+1/2}
	-
	w^{n}
	&=
	- \frac{\tau^2}{4} \calA \,  w^{n+1/2}
	- \frac{\tau}{2}\, \calA u^n  
	+ \frac{\tau}{2}\, f^n 	
	- \frac{\tau}{2}\,  \calB^*  \lambda^{n+1}
\end{align*}
and insert this in \eqref{eq:IMEX_damp_d}  to obtain
\begin{align*}
	w^{n+1}
	-
	w^{n+1/2}
	&=
	- \frac{\tau^2}{4} \calA\, w^{n+1/2} 
	- \frac{\tau}{2}\, \calA u^n  
	+ \frac{\tau}{2}\, f^{n+1} 
	- \calB^* \big( \tfrac{\tau}{2}\, \lambda^{n+1} + \mu^{n+1} \big).
\end{align*}
Using \eqref{eq:IMEX_damp_c}, then gives
\begin{align*}
	&u^{n+1} - 2 u^n + u^{n-1}	\\
	&\quad= \tau\, \big( w^{n+1/2} - w^{n} \big) + \tau\, \big( w^{n} - w^{n-1/2} \big) \\
	&\quad=
	-  \frac{\tau^2}{4} \calA\, \big(\tau w^{n+1/2} + \tau w^{n-1/2} \big) 
	- \frac{\tau^2}{2}\, \calA \big(u^n + u^{n-1}\big)  
	+ \tau^2 f^{n} 
	- \tau \calB^* \big( \tfrac{\tau}{2}\, \lambda^{n+1} + \tfrac{\tau}{2}\, \lambda^{n} + \mu^{n} \big) \\
	&\quad=
	- \frac{\tau^2}{4} \calA\, \big( u^{n+1}+ 2 u^n + u^{n-1} \big)  
	+ \tau^2 f^{n} 
	- \tau \calB^* \big( \tfrac{\tau}{2}\, \lambda^{n+1} + \tfrac{\tau}{2}\,   \lambda^{n} + \mu^{n} \big).
\end{align*}
This immediately yields the equivalence for the part in $\Vo$. On the other hand, the constraints ensure equivalence for the parts in $\Vc$, which gives the first part. For the initial step the very same considerations give the claim.
\end{proof}
\begin{remark}
To the best of our knowledge, the damped case cannot be formulated as a two-step scheme, since one cannot eliminate $w^n$ anymore. However, in the damping-free case, formulation~\eqref{eq:IMEX_two_step} is computationally much more attractive. Moreover, the two-step formulation does not need the assumption that $\cHo$ is a proper subspace of $\cH$.
\end{remark}
%
%
\subsection{Convergence analysis}
\label{sec:CN:convergence}
As first main result of this paper, we turn to the error analysis and show second-order convergence of the proposed scheme~\eqref{eq:IMEX_damp}. We emphasize that the following error bounds is valid for both the damped and damping-free case, i.e., for both schemes~\eqref{eq:IMEX_damp} and~\eqref{eq:IMEX_two_step}.
\begin{theorem}[Second-order convergence, damped IMEX method] \label{thm:error_damp_IMEX}
Consider Assumptions~\ref{assB}, \ref{assA}, and~\ref{assD} as well as consistent initial data $u^0 = \uo^0 + \calB^-g(0)$ with $\uo^0$ being an element of the domain of~$\calA_{\ker}$ and $w^0 = w_\text{ker}^0 + \calB^-\dot g(0)$ with $w_\text{ker}^0\in\Vo$. Let the solution~$u$ to~\eqref{eq:PDAE} and the right-hand sides satisfy 
\begin{align*}
	u \in C^3([0,T], \V) \cap C^4([0,T], \cH),\quad
	f(\,\cdot\,, u(\cdot)) \in C^2([0,T], \cH),\quad
	g\in C^4([0,T], \Q^*). 
\end{align*}
Then the scheme~\eqref{eq:IMEX_damp} is convergent of second order with the error estimate
\[
	\|	u(\tn{n}) - u^n \|_{\V}
	+
	\|	\dot u(\tn{n}) -w^n \|_{\cH}
	\lesssim
	\tau^2,
\]
where the hidden constant is independent of $\tau$ and $n$.
\end{theorem}
To prove second-order convergence, we use once more the decomposition~$\V = \Vo \oplus \Vc$ from~\eqref{eq:decompV} as discussed in Section~\ref{sec:PDAE:prelim}. Due to Remark~\ref{rem:consistent} and the assumed consistency of the initial data, we only have to bound the errors in $\Vo$ and $\cHo$,  
\begin{align*}
	u(\tn{n}) - u^n 
	= \big( \uo(\tn{n}) - \uo^n \big) + \big( \uc(\tn{n})  - \uc^n \big) 
	= \uo(\tn{n}) - \uo^n. 
\end{align*}
We introduce $w(t) = \dot u(t)$ with its unique decomposition $w = \wo + \wc$ as before. Note that this implies $\wc (\tn{n}) = \dot g^n \coloneqq \dot g(\tn{n})$. With this, we obtain by~\eqref{eq:IMEX_damp_e}, 
\begin{align*}
	w(\tn{n}) - w^n  
	&= \big( \wo (\tn{n}) - \wo^n \big) + \big( \wc (\tn{n})  - \wc^n \big)	\\
	&= \wo (\tn{n}) - \wo^n + \calB^-\bigl( \dot g^n - \tfrac{1}{\tau}\, ( g^{n+1/2} - g^{n-1/2}) \bigr)
	\\
	&= \wo (\tn{n}) - \wo^n + \calO(\tau^2) .
\end{align*}
We aim to show a similar error representation for the parts in the kernel of $\calB$ as in~\cite{HocL21}. 
With $\calDo$ from Section~\ref{sec:PDAE:prelim}, we introduce the operators $\Rpm$ and $\calR$ by  
\begin{subequations}  
\label{eq:def_Rpm}
\begin{align}
	\Rpm 
	&\coloneqq \begin{bmatrix}
		\id & \\ 
		& \id
	\end{bmatrix}
	\pm \frac{\tau}{2}
	\begin{bmatrix} 
		0 & -\id \\ 
		\calAo & \phantom{-} \calDo 
	\end{bmatrix}
	\colon \Vo \times \cHo \to \cHo \times \Vo^*,
	\\
	\calR 
	&\coloneqq \Rp^{-1} \Rm
	\colon \Vo \times \cHo \to \Vo \times \cHo.
\end{align}
\end{subequations}
These operators appear naturally, when interpreting the IMEX method as perturbation of the \CN method. Here, $\Rpm$ encode the structure of the implicit midpoint rule. However, this formulation is not suited for the implementation, and we refer to \cite{HocL21} for more details. The essential properties of these operators are collected in the following lemma.
\begin{lemma} \label{lem:Rpm_properties}
The operators $\Rpm$ and $\calR$ from~\eqref{eq:def_Rpm} satisfy the bounds
\begin{equation*}
	\norm{\calR}_{\Vo \times \cHo \to \Vo \times \cHo},\ 
	\norm{\Rp^{-1}}_{\cHo \times \Vo^* \to \Vo \times \cHo} 
	\ \leq\ e^{c \tau}	 
\end{equation*}
for some constant $c > 0$, independent of $\tau$. Further, it holds 
for all $x \in \cHo$
the identity
\begin{equation} \label{eq:Rpinv_relation}
	\tau \Rp^{-1} 
	\begin{bmatrix}
		- x \\ \calDo x  
	\end{bmatrix}
	=
	(\id - \calR)
	\begin{bmatrix}
		0 \\  x  
	\end{bmatrix}
\end{equation} 
and, thus, the bound
\begin{equation*}
	\Big\| \Rp^{-1} 
	\begin{bmatrix}
		- x \\ \calDo x  
	\end{bmatrix} \Big\|_{\Vo \times \cHo}
	\leq \frac{1+ e^{c \tau}}{\tau}\, \norm{ x }_{\cHo}
	\leq \frac{2 e^{c \tau}}{\tau}\, \norm{ x }_{\cHo}
	.
\end{equation*}
\end{lemma}
\begin{proof}
The bounds on $\calR,\Rp^{-1}$ follow from \cite[Lem.~2.4]{HocL21}. From the definition of $\Rpm$, we have
\begin{equation*}
	( \Rp - \Rm )	
	\begin{bmatrix}
		0 \\  x  
	\end{bmatrix}
	=
	\tau
	\begin{bmatrix}
		-x  \\  \calDo x  
	\end{bmatrix},
\end{equation*}
and multiplying with $\Rp^{-1}$ gives the claim. 
\end{proof}
For technical reasons, we introduce the following projection operator.
\begin{lemma}
	\label{lem:projection}	
	Define the orthogonal projection $\piker \colon \cH \to \cHo \subseteq \Vo^*$ via%
	\begin{equation*}
		\ip{ \piker v ,  \testker}_{\cH} = 	\ip{ v ,  \testker}_{\cH}
		\quad \text{ for all } \quad  
		\testker \in \cHo.
	\end{equation*}
	Then it holds $\norm{\piker v}_{\cH} \leq \norm{v}_{\cH}$ and, since $\Vo \subseteq \cHo$, it particularly holds
	\begin{equation*}
		\ip{ \piker v ,  \testker}_{\cH} = 	\ip{ v ,  \testker}_{\cH} 
		\quad \text{ for all } \quad  
		\testker \in \Vo.
	\end{equation*}
\end{lemma}
With these operators, we can state an error representation of the IMEX method \eqref{eq:IMEX_damp}.
\begin{lemma} \label{lem:error_equation_damp_IMEX}
Let the assumptions of Theorem~\ref{thm:error_damp_IMEX} hold. 
Then the errors in the kernel
\begin{equation}
	\errunCN{n} 
	\coloneqq  \uo(\tn{n}) - \uo^n,
	\qquad
	\errwnCN{n} 
	\coloneqq \wo (\tn{n}) - \wo^n
\end{equation}
satisfy the relation
\begin{equation} \label{eq:IMEX_first_order_error}
	\Rp 
	\begin{bmatrix}
		\errunCN{n+1} \\ \errwnCN{n+1}
	\end{bmatrix}
	=
	\Rm
	\begin{bmatrix}
		\errunCN{n} \\ \errwnCN{n}
	\end{bmatrix}
	+
	\begin{bmatrix}
		0			\\
			\piker\, S_2^n 
		\end{bmatrix}
		+
		\begin{bmatrix}
			\piker\, S_1^n 
			\\
			- \dampo \piker\, S_1^n 
		\end{bmatrix}
		+
		\begin{bmatrix}
			\piker\,\delta^n_u \\
			\piker\,\delta^n_w
		\end{bmatrix}
		+
		\begin{bmatrix}
			- \piker\, ( \delta^n_1 +  \delta^n_2 )
			\\
			\dampo 	\piker\, ( \delta^n_1 + \delta^n_2)
		\end{bmatrix}
	,
\end{equation}
where we have the two defects $	\delta^n_u$ and $\delta^n_w$ that essentially come from the trapezoidal rule and satisfy $\norm{\delta^n_u}_{\Vo} + \norm{\delta^n_w}_{\cHo} \leq C \tau^3$. For $\fex^{n}  = f(\tn{n},u(\tn{n}))$ and $\f^{n}  = f(\tn{n},u^n)$, 
the other terms are given by
\begin{alignat*}{3}
	S_1^n 
	&=
	\frac{\tau^2}{4}\, \big( \fex^{n} - \fex^{n+1}\big)
	-
	\frac{\tau^2}{4}\, \big( \f^{n} - \f^{n+1}\big) ,
	\qquad
	&& S_2^n 
	= & &
	\frac{\tau}{2}\, \big( \fex^{n+1} + \fex^{n}\big)
	-
	\frac{\tau}{2}\, \big( f^{n+1} + f^{n}\big) ,
	\\
	\delta^n_1
	&= \frac{\tau^2}{4}\, \big( \fex^{n} - \fex^{n+1}\big) ,
	&&\delta^n_2
	= & &
	\frac\tau2\, \big( \wc^{n+1} +  \wc^{n} \big) - \tau \wc^{n+1/2}  .
\end{alignat*}
\end{lemma}
Before proofing this lemma, we briefly sketch the proof of our main result.
\begin{proof}[Proof of Theorem~\ref{thm:error_damp_IMEX}]
We apply $\Rp^{-1}$ to \eqref{eq:IMEX_first_order_error} and resolve the error recursion to 
\begin{align*} 
	\begin{bmatrix}
		\errunCN{n+1} \\ \errwnCN{n+1}
	\end{bmatrix}
	=
	\RpinvRm^{n+1}
	\begin{bmatrix}
		\errunCN{0} \\ \errwnCN{0}
	\end{bmatrix}
	&+
	\sum\limits_{j=0}^n		
	\RpinvRm^{n-j}
\Rp^{-1}
	\begin{bmatrix}
		0			\\
			\piker\, S_2^j 
		\end{bmatrix}
		+
		\RpinvRm^{n-j}
	\Rp^{-1}
		\begin{bmatrix}
			\piker\,\delta^j_u \\
			\piker\,\delta^j_w
		\end{bmatrix}
\\
&+
\sum\limits_{j=0}^n		\RpinvRm^{n-j}
\Rp^{-1}
	\begin{bmatrix}
			\piker\, S_1^j 
			\\
			- \dampo \piker\, S_1^j 
		\end{bmatrix}
		+
		\RpinvRm^{n-j}
		\Rp^{-1}
		\begin{bmatrix}
			- \piker\, ( \delta^j_1 +  \delta^j_2 )
			\\
			\dampo 	\piker\, ( \delta^j_1 + \delta^j_2)
		\end{bmatrix}
		.
\end{align*}
With the bounds in Lemma~\ref{lem:Rpm_properties}, we observe that the first term accounts for the classical linear stability term, the second term is estimated with the Lipschitz property of $f$, and the third term simply yields a sum over $\calO(\tau^3)$ terms. The fourth and fifth term, on the other hand, need to be taken care of because their upper entries cannot be estimated in $\Vo$.

By the application of $\Rp^{-1}$, however, we can use the relation \eqref{eq:Rpinv_relation} and the term $(\id - \calR )$ allows for summation by parts. This enables us to replace the spatial regularity in $\Vo$ by an additional (discrete) time derivative that maintains the order of the defects and the stability term. We refer to the proofs of~\cite[Th.~2.9 and 3.3]{HocL21}, which then also apply to the present case.
\end{proof} 
Hence, it only remains to prove the error recursion. Here, we have to take into account all the additional contributions arising from the constraint.
\begin{proof}[Proof of Lemma~\ref{lem:error_equation_damp_IMEX}]
	The proof is divided into three parts.
	First, we derive a suitable formulation for the approximations $\uo^n$ and $\wo^n$. In the second step, we derive a similar form for the exact solution.
	Lastly, we combine the representations to obtain the claimed formulation. 	
	
	(a) From~\eqref{eq:IMEX_damp} we obtain two representations of $w^{n+1/2}$, namely 
	\begin{align*}
		w^{n+1/2} = \frac12\, \big( w^{n+1} + w^{n} \big) +  \frac{\tau}{4}\, \big( f^{n} - f^{n+1}\big)
		+
		\frac12\, \calB^* \mu^{n+1} 
	\end{align*}
	as well as
	\begin{align*}
		w^{n+1/2}
		&= 
		-\frac{\tau}{2}\, \damp w^{n+1/2} - \frac{\tau}{4}\, \calA\, \tau w^{n+1/2}
		+
		w^n - \frac{\tau}{2}\, \calA u^n  
		- 
		\frac{\tau}{2}\, \calB^*  \lambda^{n+1}
		+ \frac{\tau}{2}\, f^n.
	\end{align*}
	Note that both equations are stated in~$\V^*$. 
	From the first equation, we conclude that 
	\begin{equation*}
		u^{n+1} 
		= u^n + \tau w^{n+1/2}
		= u^n 
		+
		\frac{\tau}{2}\, \big( w^{n+1} + w^{n} \big) +  \frac{\tau^2}{4}\, \big( f^{n} - f^{n+1}\big)
		+
		\frac{\tau}{2}\, \calB^* \mu^{n+1}  .
	\end{equation*}
	Moreover, the two representations of $w^{n+1/2}$ imply that 
	\begin{align*}
		&
		\frac12\, \big( w^{n+1} + w^{n} \big) +  \frac{\tau}{4}\, \big( f^{n} - f^{n+1}\big)
		+
		\frac12\, \calB^* \mu^{n+1} 
		\\
		&\qquad= 
		-\frac{\tau}{2}\, \damp w^{n+1/2} - \frac{\tau}{4}\, \calA \big(u^{n+1} - u^n \big)
		+
		w^n - \frac{\tau}{2}\, \calA u^n  
		- 
		\frac{\tau}{2}\, \calB^*  \lambda^{n+1}
		+ \frac{\tau}{2}\, f^n  .
	\end{align*}
	We rearrange this to
	\begin{align*}
		w^{n+1} -  w^{n}  
		&= 
		-\tau \damp w^{n+1/2}
		- \frac{\tau}{2}\, \calA \big( u^{n+1} + u^n \big)
		+  \frac{\tau}{2}\, \big( f^{n+1} + f^{n}\big)
		- 
		\tau  \calB^* \lambda^{n+1}
		-	
		\calB^* \mu^{n+1}  .
	\end{align*}
	Thus, we have established
	\begin{equation} 
		\label{eq:u_n_ker_1}
		\begin{aligned}
			u^{n+1} 
			&= u^n 
			+
			\frac{\tau}{2}\, \big( w^{n+1} + w^{n} \big) +  \frac{\tau^2}{4}\, \big( f^{n} - f^{n+1}\big)
			+
			\frac{\tau}{2}\, \calB^* \mu^{n+1}  ,
			\\
			w^{n+1}
			&= 
			w^n 
			- \frac{\tau}{2} \calA \big( u^{n+1} + u^n \big)
			+  \frac{\tau}{2} \big( f^{n+1} + f^{n}\big)
			-\tau \damp w^{n+1/2}
			- 
			\tau  \calB^*  \lambda^{n+1}
			-	\calB^* \mu^{n+1}.
		\end{aligned}
	\end{equation}
	Restricting~\eqref{eq:IMEX_damp_d} to test functions in $\Vo$, we obtain in $\Vo^*$ the equation 
	\begin{align*}
		\wo^{n+1/2} 
		&= 
		\frac12\, \big( \wo^{n+1} + \wo^{n} \big)  
		+
		\frac{\tau}{4}\, \piker \big( f^{n} -  f^{n+1} \big)
		+
		\frac12\, \piker \big( \wc^{n+1} + \wc^{n} \big) - \piker\wc^{n+1/2} . 
	\end{align*}
	Similarly, testing \eqref{eq:u_n_ker_1} with functions in $\Vo$, it holds in $\Vo^*$, 
	\begin{align*}
		\uo^{n+1} 
		&= \uo^n 
		+
		\frac{\tau}{2}\, \big( \wo^{n+1} +  \wo^{n} \big)
		+ 
		\frac{\tau^2}{4}\, \piker  \big( f^{n} - f^{n+1}\big)
		+ \piker  r_1,
		\\
		\wo^{n+1}
		&= 
		\wo^n 
		- \frac{\tau}{2}\, \calAo\, \big( \uo^{n+1} + \uo^n \big)
		+  \frac{\tau}{2}\, \piker \big( f^{n+1} + f^{n}\big)
		-\tau \dampo \wo^{n+1/2}
		+ \piker  r_2
		\\
		&= 
		\wo^n 
		- \frac{\tau}{2}\, \calAo\, \big( \uo^{n+1} + \uo^n \big)
		- \frac{\tau}{2}\, \dampo \big( \wo^{n+1} + \wo^{n} \big)
		+ \frac{\tau}{2}\, \piker \big( f^{n+1} + f^{n}\big) \\
		&\hspace{2.4cm}-  
		\frac{\tau^2}{4}\, \dampo\, \piker \big( \f^{n} - \f^{n+1}\big) 
		+ 
		\piker  r_2
		-
		\dampo\, \piker  \delta^n_2 
	\end{align*}
	with remainders
	\begin{align*}
		r_1 
		= \uc^n  - \uc^{n+1}  
		+
		\frac{\tau}{2}\, \big( \wc^{n+1} +  \wc^{n} \big), \qquad
		r_2 
		= \wc^n  - \wc^{n+1} 	- \tau \damp \wc^{n+1/2}.
	\end{align*}
	Using that 
	$
	\uc^n  - \uc^{n+1}  = \calB^{-}\bigl(
	g^{n} - g^{n+1} 
	\bigr) = 
	- \tau \wc^{n+1/2} 
	$,
	we get $r_1 = \delta^n_2$. \medskip
	
	(b) Noting that the exact solution satisfies
	\begin{align*}
		\dot{u}(t) 
		= w(t), \qquad 
		\dot{w}(t)
		= - \damp w(t) - \calA u(t) + f(t,u(t))+ \calB^* \lambda(t), 
	\end{align*}	
	we conclude together with $\ucdot(t) = \wc(t)$ that 
	\begin{align*}
		\uodot(t) 
		= \wo(t), \quad  
		\wodot(t)
		= - \dampo \wo(t) 
		- \calAo \uo(t)  
		+ \piker f(t,u(t))
		- \damp \wc(t)
		- 	\ucddot(t) .
	\end{align*}	
	We introduce the notation $\uexo^n = \uo(\tn{n})$, $\uexc^n = \uc(\tn{n})$
	and further denote with $\delta^n_{u}$, $\delta^n_{w,1}$ the defects from the trapezoidal rule, which satisfy $\norm{\delta^n_{u}}_{\V} + \norm{\piker \delta^n_{w,1}}_{\cH} \leq C \tau^3$. 
	We can thus conclude that the exact solution satisfies
	\begin{align*}
		\uexo^{n+1}
		&= \uexo^{n} +
		\frac{\tau}{2}\, \big( \wexo^{n+1} + \wexo^{n} \big)
		+ 
		\frac{\tau^2}{4} \piker \big( \fex^{n} - \fex^{n+1}\big)
		+
		\piker \delta^n_{u} - \piker \delta^n_1 
		\\
		\wexo^{n+1}
		&= 
		\wexo^n 
		-
		\frac{\tau}{2}\, \calAo\, \big( \uexo^{n+1} + \uexo^n \big)
		-
		\frac{\tau}{2}\, \dampo	\big( \wexo^{n+1} + \wexo^{n} \big) 
		+
		\frac{\tau}{2}\, \piker \big( \fex^{n+1} + \fex^{n}\big)
		\\
		&\ -
		\frac{\tau}{2}\,  \damp \big( \wexc^{n+1} + \wexc^{n} \big) 
		- 
		\frac{\tau}{2} \big( \ucddot^{n+1} + \ucddot^{n} \big)
		-
		\frac{\tau^2}{4} \dampo \piker \big( \fex^{n} - \fex^{n+1}\big) 
		+
		\piker \delta^n_{w,1}
		+
		\dampo \piker \delta^n_1,
	\end{align*}
	where the defect $\delta^n_1$ is defined in the assertion of Lemma~\ref{lem:error_equation_damp_IMEX} and is artificially added to compensate the difference of $\fex^{n}$. \medskip
	
	(c) Recalling the definitions of $\errunCN{n}$ and $\errwnCN{n}$, we obtain with parts (a) and (b), 
	\begin{align*}
		\errunCN{n+1} 
		&= \errunCN{n} +
		\frac{\tau}{2} \big( \errwnCN{n+1} +  \errwnCN{n} \big)
		+ 
		\piker S_1^n 
		+
		\piker \delta^n_{u} -
		\piker \big( \delta^n_1 + \delta^n_2 \big),
		\\
		\errwnCN{n+1}
		&= 
		\errwnCN{n}
		-
		\frac{\tau}{2}\, \calAo\, \big( \errunCN{n+1} + \errunCN{n} \big)
		-
		\frac{\tau}{2}\,  \dampo	\big( \errwnCN{n+1} + \errwnCN{n} \big) 
		+
		\piker S_2^n - \dampo \piker S_1^n 
		\\
		&\quad -
		\frac{\tau}{2}\, \damp \big( \wexc^{n+1} + \wexc^{n} \big) 
		- 
		\frac{\tau}{2} \big( \ucddot^{n+1} + \ucddot^{n} \big)
		+
		\piker \delta^n_{w,1}
		+
		\dampo \piker \delta^n_1 
		-
		\piker r_2 + \dampo \piker  \delta^n_2 
		\\
		&= 
		\errwnCN{n}
		-
		\frac{\tau}{2}\, \calAo\, \big( \errunCN{n+1} + \errunCN{n} \big)
		-
		\frac{\tau}{2}\, \dampo \big( \errwnCN{n+1} + \errwnCN{n} \big) 
		+
		\piker S_2^n - \dampo \piker S_1^n 
		\\
		&\quad 
		+
		\piker \delta^n_{w,1}
		+
		\piker \delta^n_{w,2}
		+
		\dampo \piker \big( \delta^n_1 +  \delta^n_2 \big) 
	\end{align*}
	in $\Vo^*$, where $\delta^n_{w,2} $ satisfies
	\begin{align*}
		\norm{\delta^n_{w,2} }_{\cH} 
		&= \big\| \frac{\tau}{2}\,  \damp	\big( \wexc^{n+1} + \wexc^{n} \big) 
			+ \frac{\tau}{2} \big( \ucddot^{n+1} + \ucddot^{n} \big)
			+ r_2	\big\|_{\cH} 
		\\
		&= \big\| \frac{\tau}{2}\,  \damp	\big( \wexc^{n+1} + \wexc^{n} \big) 
			+ \frac{\tau}{2} \big( \ucddot^{n+1}  + 	\ucddot^{n} \big)
			- \big(	\wc^{n+1} - 	\wc^n 	+ \tau \damp \wc^{n+1/2} \big) \big\|_{\cH} 
		\\
		&= \big\| \tau  \damp \bigl( 
			\wc^{n+1/2} - \tfrac{1}{2}\, ( \wexc^{n+1} + \wexc^{n} )  \bigr) 
			- 
			\bigl( \wc^{n+1} - 	\wc^n  - 	\tfrac{\tau}{2}\, ( \ucddot^{n+1}  + 	\ucddot^{n} ) \bigr) \big\|_{\cH}  
		\\
		&\leq C\, \tau^3 .
	\end{align*}
	We finally define 
	$\delta^n_{w} = \piker ( \delta^n_{w,1}  + \delta^n_{w,2} ) $
	and use the definition of $\Rpm$
	to obtain the relation~\eqref{eq:IMEX_first_order_error}.
\end{proof}
\begin{remark} \label{rem:H_eq_H_ker}
Let us finally comment on the case that $\cH = \cHo$ and the modified method, where $w^{n+1}$ is computed by~\eqref{eq:variant_H_Hker_IMEX}.
We first note that the same proof as in Lemma~\ref{lem:eq_two_step_undamped}
yields the equivalence of the two schemes if we omit the $\calB^* \mu$-terms in all calculations.
Similarly, in the proof of Lemma~\ref{lem:error_equation_damp_IMEX}, the computations until \eqref{eq:u_n_ker_1} are fully analogous, again omitting the $\calB^* \mu$-terms.
From that on, we set 
$\wc^{n+1} = \tfrac{1}{\tau}\, \calB^-\big( g^{n+3/2} - g^{n+1/2} \big)$ 
and $\wo^{n+1} = w^{n+1} - \wc^{n+1}$ (even though $\wo^{n+1}$ might not be in the kernel of $\calB$, but in $\cH = \cHo$). 
Then, the very same calculations yield the error recursion in~\eqref{eq:IMEX_first_order_error}. 
The estimates of Lemma~\ref{lem:Rpm_properties} hold on $\Vo \times \cH$ and it is sufficient that $\wo^{n+1} \in \cH$ without any knowledge on the constraint. In particular, we have the same error bound as in Theorem~\ref{thm:error_damp_IMEX} under the same assumptions.
\end{remark}
%
%
\section{A Gautschi-type Integrator}
\label{sec:Gautschi}
This section is devoted to the extension of Gautschi-type integrators to constrained systems. For this, we restrict ourselves to the damping-free free case, i.e., we set $\damp = 0$. Moreover, we consider the stronger assumption on the nonlinearity, i.e., Assumption~\ref{assRHS}~(a). This is related to the fact that in the proof of Theorem~\ref{th:Gautschi:convergence} a stronger notation of stability is required to prevent order reduction.

As in the previous section, we consider a uniform partition of~$[0,T]$ with constant time step size~$\tau$ and discrete time points~$\tn{n}=n\, \tau$. 
%
%
\subsection{Review of the unconstrained case} 
\label{sec:Gautschi:unconstrained}
We recall the concept of Gautschi-type integrators for wave-type systems with bounded operators as well as their extension to unbounded operators. For this, an interpretation of trigonometric functions applied to an operator is needed. 
%
%
\subsubsection*{Bounded operators} 
We recall solution strategies for conservative second-order systems as discussed in~\cite{HocL99,Gri02}. Considering the corresponding first-order formulation, we consider the equations 
\begin{align*}
	\dot u(t) - w(t)
	&= 0, \\
	\dot w(t) + \calA u(t)
	&= f(t, u) 
\end{align*}
with a bounded operator $\calA$, which also has a bounded inverse. In terms of operator matrices, this system reads
\begin{align}
\label{eq:PDE:firstOrdNoDamp}
	\begin{bmatrix} \dot u(t) \\ \dot w(t) \end{bmatrix}
	+ \begin{bmatrix} 0 & -\id \\ \calA & 0 \end{bmatrix}
	\begin{bmatrix} u(t) \\ w(t) \end{bmatrix}
	= \begin{bmatrix} 0 \\ f(t, u) \end{bmatrix}, 
\end{align}
which we abbreviate by~$\dot z + \calZ z = \tilde f$ for the unknown~$z\coloneqq [u;w]$, where we use the standard Matlab notation for the concatenation of vectors. The initial data is given by~$z(0) = [u^0;w^0]$. By an application of the variation-of-constants formula~\cite[Ch.~4]{Paz83}, we conclude 
\begin{align*}
	z(t)
	= e^{-t\calZ}z(0) + \int_0^t e^{-(t-s)\calZ} \tilde f(s,z) \ds. 
\end{align*}
Under the assumption that $\calA$ is bounded, the exponential is well-defined and has the form  
\begin{equation}
\label{eq:semigroupNoDamp}
	e^{-t\calZ}
	= \begin{bmatrix} \cos(t \Omega) & \Omega^{-1} \sin(t \Omega) \\ 
		-\Omega\, \sin(t \Omega) & \cos(t \Omega) \end{bmatrix}
	= \begin{bmatrix} \cos(t \Omega) & t\, \sinc(t \Omega) \\ 
		-t\calA\, \sinc(t \Omega) & \cos(t \Omega) \end{bmatrix} 
\end{equation}
with $\Omega \coloneqq \calAh$; cf.~\cite[Sect.~3.1]{Gri02}. Although written in terms of the square root of~$\calA$, the functions $\sinc$ and $\cos$ given by the Taylor series 
\begin{equation}
\label{eq:TaylorSincCos}
	\sinc x 
	= \sum_{k=0}^{\infty} \frac {(-x^2)^k}{(2k+1)!}, \qquad 
	\cos x
	= \sum_{k=0}^{\infty} \frac {(-x^2)^k}{(2k)!}
\end{equation}
only depend on even powers of the input such that $\sinc(t\Omega)$ and $\cos(t\Omega)$ are well-defined even if~$\Omega$ is not. 
By formula~\eqref{eq:semigroupNoDamp}, the solution~$u$ being equal to the first component of~$z$ can be written as 
\begin{align}
\label{eq:slnUnconstrained}
	u(t)
	= \cos(t \Omega)\, u^0 + t\, \sinc(t \Omega)\, w^0
	+ \int_0^t (t-s)\, \sinc((t-s)\Omega)\, f(s,u) \ds.
\end{align}
%
For further calculations, we commemorate the identity 
\[
	\frac{1-\cos x}{x^2}
	= \int_0^1 x^{-1} \sin \big( (1-\eta)\, x \big) \deta
	= \int_0^1 (1-\eta)\, \sinc \big( (1-\eta)\, x \big) \deta. 
\]
Inserting $x=\tau\,\Omega$ and using a substitution to solve the integral, we get 
\[
	\calA^{-1} \big[ 1 - \cos(\tau\Omega) \big]
	= \int_0^1 \tau^2\, (1-\eta)\, \sinc \big((1-\eta)\, \tau \Omega \big) \deta  
	= \int_0^\tau\, (\tau-s)\, \sinc\big((\tau-s)\, \Omega \big) \ds. 
\]
With the solution formula~\eqref{eq:slnUnconstrained}, we conclude for the solution of~\eqref{eq:PDE:firstOrdNoDamp} that
\begin{align*}
	u(t+\tau) + u(t-\tau)
	= 2\, \cos(\tau \Omega)\, u(t) 
	+ 2\int_0^\tau (\tau-s)\, \sinc((\tau-s)\, \Omega)\, f(s,u(s)) \ds.
\end{align*}
For a constant right-hand side, i.e., for~$f(t,u)\equiv f$, this further simplifies to 
\begin{align*}
	u(t+\tau) + u(t-\tau)
	&= 2\, \cos(\tau \Omega)\, u(t) + 2\, \calA^{-1} \big[ 1 - \cos(\tau\Omega) \big]\, f \\
	&= 2\, \cos(\tau\Omega) \big[ u(t) - \calA^{-1} f\big] + 2\, \calA^{-1} f. 
\end{align*}
%
In the semi-linear case, we may approximate the integral term by the simplest integration scheme, i.e., we replace~$f(t,u)$ by~$f(\tn{n},u^n)$ with~$u^n\approx u(\tn{n})$. Inserting this in the previous calculation yields the integration scheme 
\begin{align}
	\label{eq:Gautschi:noDamp:noConstraint}
	u^{n+1} 
	= 2\, \cos(\tau\Omega) \big[ u^n - \calA^{-1} f(\tn{n},u^n)\big] 
	+ 2\, \calA^{-1} f(\tn{n},u^n) - u^{n-1}.
\end{align}
Following~\cite{HocL99}, this may also be written in the alternative form 
\[
	u^{n+1} - 2u^n + u^{n-1}
	= 2\, \calA^{-1} \big(1 - \cos(\tau\Omega)\big)\, \big[ f(\tn{n},u^n) - \calA u^n\big]
\]
or, due to $2\, (1-\cos(x)) / x^2 = \sinc^2(\tfrac x2)$, as 
\[
	u^{n+1} - 2u^n + u^{n-1}
	= \tau^2 \sinc^2(\tfrac\tau2 \Omega)\, \big[ f(\tn{n},u^n) - \calA u^n \big].
\]
%
%
\subsubsection*{Extension to unbounded operators}
So far, we have assumed that the operator~$\calA$ is bounded. Hence, the exponential as well as~$\cos(t\Omega)$ are well-defined, e.g., using the Taylor series.  
In the here considered PDAE setting, $\calA$ usually denotes an unbounded differential operator such that a different approach is needed. For this, consider the situation as introduced in Section~\ref{sec:PDAE:prelim} with the Hilbert spaces~$\V, \cH, \Q$ and an operator~$\calA$ as in Lemma~\ref{lem:existenceLinUnconstrained}. 

In order to make the numerical scheme~\eqref{eq:Gautschi:noDamp:noConstraint} also valid for unbounded operators, we need an interpretation of the trigonometric terms. For this, we assume for a minute that $\calA$ is elliptic. 
In this case, we can interpret~$\cos(t\Omega)\,v$ as the solution of the linear and homogeneous equation 
\begin{align}
	\label{eq:cosOmega}
	\ddot u(t) + \calA u(t) = 0, \qquad
	u(0)=v,\ 
	\dot u(0)=0.
\end{align}
Moreover, we can understand~$t\, \sinc(t \Omega)\,v$ as the solution of 
\begin{align}
	\label{eq:sincOmega}
	\ddot u(t) + \calA u(t) = 0, \qquad
	u(0)=0,\ 
	\dot u(0)=v.
\end{align}
Note that both systems are well-posed by Lemma~\ref{lem:existenceLinUnconstrained}. 
%
%
\subsection{An implicit solution formula}
\label{sec:Gautschi:slnformula}
We come back to the constrained case and consider the PDAE~\eqref{eq:PDAE}. In the linear case, i.e., if~$f(t,u)\equiv f(t)$, we can explicitly write down the solution using the decomposition of~$\V$ from~\eqref{eq:decompV}, the first-order formulation, and the variation-of-constants formula; cf.~\cite{EmmM13}. This solution formula is also applicable to the general semi-linear case, leading to an implicit formula. As in the unconstrained case, quadrature rules then yield numerical schemes. 

Following~\eqref{eq:decompV}, we decompose the solution into~$u = \uo + \uc$ with~$\uo\colon[0,T] \to \Vo$ and~$\uc = \calB^-g(t)\in\Vc \colon[0,T] \to \Vc$. Then the restriction of~\eqref{eq:PDAE:a} to test functions in~$\Vo$ yields 
\begin{align}
	\label{eq:uker}
	\uoddot + \calAo \uo
	= f(t, \uo+\uc) - \ucddot\qquad\text{in } \Vo^*.
\end{align}
Note that there is no term~$\calA \uc$ on the right-hand side due to the particular definition of~$\Vc$. Since this is an unconstrained system with an elliptic differential operator~$\calAo$, the variation-of-constants formula can be applied to the corresponding first-order formulation.  
With the operator $\piker\colon \cH \to \cHo$ introduced in Lemma~\ref{lem:projection}, which is simply based on the restriction of test functions, we conclude that
\begin{align}
	\label{eq:slnFormulaPDAE}
	u(t) 
	&= \uc(t) + \uo(t) \\
	&= \calB^-g(t) + \cos(t\Omo) \uo(0) + t \sinc(t\Omo) \uodot (0) \notag \\
	&\hspace{1cm}+ \int_0^t (t-s)\, \sinc((t-s)\, \Omo)\, \piker \big[ f(s, \uo(s)+\calB^-g(s)) - \calB^-\ddot g(s) \big] \ds. \notag
\end{align}
Here, we use the notation~$\uo(0) \coloneqq u^0 - \calB^-g(0)$, $\uodot(0) \coloneqq w^0 - \calB^- \dot g(0)$ and~$\Omo \coloneqq \calAoh$. Note that this solution formula is {\em implicit} in the sense that the solution also appears on the right-hand side within the integral term. Nevertheless, this representation is the basis for the derivation of a Gautschi-type integrator as we will see in the following.  
%
%
\subsection{A Gautschi-type integration scheme}
\label{sec:Gautschi:scheme}
Considering the solution formula~\eqref{eq:slnFormulaPDAE} at times $\tn{n+1}$ and $\tn{n-1}$ (and omitting the restriction operator~$\piker$ for simplicity), we get 
\begin{align*}
	u(\tn{n+1}) &+ u(\tn{n-1})
	= \calB^-g(\tn{n+1}) + \calB^-g(\tn{n-1}) + 2\cos(t\Omo) \big[u(\tn{n}) - \calB^-g(\tn{n}) \big] \\
	&\qquad+ \int_0^\tau (\tau-s)\, \sinc((\tau-s)\, \Omo)\, \big[ f(\tn{n}+s, u(\tn{n}+s)) - \calB^-\ddot g(\tn{n}+s) \\ 
	&\hspace{6.3cm}+ f(\tn{n}-s, u(\tn{n}-s)) - \calB^-\ddot g(\tn{n}-s) \big] \ds. 
\end{align*}
Now, we fix the nonlinear term in the integral by its value on the left limit, i.e., by~$f(\tn{n}, u^n) - \calB^-\ddot g^n$. With
\[
	b^n 
	\coloneqq \calAo^{-1}\, \piker \big[ f(\tn{n}, u^n) - \calB^-\ddot g^n \big],
\]
we obtain by a calculation similarly as in Section~\ref{sec:Gautschi:unconstrained} that  
\begin{align*}
	u^{n+1}
	= \calB^-g^{n+1} + \cos(\tau\Omo) (u^n - \calB^- g^n) 
	+ \tau \sinc(\tau\Omo) (\dot u^n - \calB^- \dot g^n) 
	- \cos(\tau\Omo)\, b^n + b^n.
\end{align*}
Considering the analog formula for~$u^{n-1}$, we obtain the scheme 
\begin{align}
\label{eq:Gautschi:PDAE}
	u^{n+1} 
	%
	%
	= \calB^-g^{n+1}
	- ( u^{n-1} - \calB^-g^{n-1}) 
	+ 2 \cos(\tau\Omo) ( u^n - \calB^-g^n - b^n) + 2\, b^n. 
\end{align}	
\begin{remark}
The proposed numerical scheme~\eqref{eq:Gautschi:PDAE} is in line with Section~\ref{sec:Gautschi:unconstrained}. This means that we recover the scheme~\eqref{eq:Gautschi:noDamp:noConstraint} if we leave out the constraint. 
\end{remark}
We would like to emphasize that the Gautschi-type integrator~\eqref{eq:Gautschi:PDAE} is a two-step scheme. Hence, we need to provide approximations~$u^0$ and~$u^1$ to apply the scheme. For~$u^0$ we can directly insert~$u(0)$. For the second value, we need to ensure~$u^1 = u(\tau) + \calO(\tau^3)$ in order to allow a second-order accuracy of the scheme; see Theorem~\ref{th:Gautschi:convergence} below. In fact, due to the second order structure of the problem, in the error accumulation we indeed loose one order of $\tau$ on the initial value and two orders in the defect. As an example, one may use a Taylor series approach in combination with~\eqref{eq:uker}, leading to  
\[
  	u^1
	\coloneqq u^0 + \tau\, w^0 + \tfrac12 \tau^2\, \big[ f(0, u^0) - \calAo (u^0-\calB^-g^{0}) \big]. 
\]
Note, however, that this requires a sufficiently regular $u^0$. 
%
%
\subsection{Convergence analysis}
\label{sec:Gautschi:convergence}
As preparation for the upcoming convergence proof, we summarize some facts on $\cos(t \Omega)$, $\sinc(t \Omega)\colon [0,\infty) \to \calL(\calV)$. Recall that these time-dependent operators should be understood as part of the solution mapping; see~\eqref{eq:cosOmega} and~\eqref{eq:sincOmega}.
\begin{lemma}
\label{lem:properties:cos}
Let $\calA\in \calL(\calV,\calV^*)$ satisfy the assumptions of Lemma~\ref{lem:existenceLinUnconstrained} and set~$\Omega=\calAh$. 
Then the operators $\cos(t \Omega)$, $t\sinc(t\Omega)$, and~$t\calA\sinc(t\Omega)$ are extendable to $t \in \R$ and it holds that 
\begin{enumerate}[label=\alph*)]
\item $\cos(t \Omega)$ is bounded as operator from $\calV \to \calV$ and $\calH \to \calH$ for every $t \in \R$, \label{item:properties:cos_1} \\[-0.75em]
\item $\cos((t_1 + t_2)\,\Omega) = \cos(t_1 \Omega)\cos(t_2 \Omega) - t_1\sinc(t_1 \Omega)\, t_2\calA \sinc(t_2 \Omega)$ for every $t_1,t_2 \in \R$,\label{item:properties:cos_2} \\[-0.75em]
\item $\id = \cos^2(t\Omega) + t\sinc(t \Omega)\,t\calA\sinc(t \Omega)$ for every $t \in \R$. 
\label{item:properties:cos_3}
\end{enumerate}
\end{lemma}
\begin{proof} 
\ref{item:properties:cos_1} For non-negative $t$ and as an operator in~$\calL(\calV)$ the assertion follows immediately from estimate~\eqref{eq:stabEstimate:noConst}. In particular, we get the bound~$1$ in the symmetric case and $e^{\kappa t/2}$ in the general case with $\calA_2 \neq 0$. For $u(t) = \cos(t \Omega)\, u^0$, the function $y(t) \coloneqq \int_0^t u(s) \ds$ solves~\eqref{eq:PDE:firstOrdNoDamp} with homogeneous right-hand side and initial values $y(0) = 0$ and $\dot{y}(0) = u^0$. In particular, it holds that $\dot{y} = u$. Therefore, the assertion follows again from~\eqref{eq:stabEstimate:noConst} by extending the operator from~$\calL(\calV,\calH)$ to $\calL(\calH)$. Finally, for negative~$t$ we notice that system~\eqref{eq:PDE:firstOrdNoDamp} with homogeneous right-hand side is invariant under time reversion such that $\cos(-t \Omega)= \cos(t \Omega)$. \medskip

\ref{item:properties:cos_2} Using the semigroup~\eqref{eq:semigroupNoDamp} we have
\begin{align*}
	\begin{bmatrix}
	\cos((t_1+t_2)\,\Omega)\, u^0\\ *
	\end{bmatrix} 
	&= \big(e^{-t_1\calZ} e^{-t_2\calZ}\big)
	\begin{bmatrix} u^0\\ 0	\end{bmatrix}\\ 
	&= \begin{bmatrix} \cos(t_1\Omega)\cos(t_2\Omega) u^0 - t_1\sinc(t_1 \Omega)\, t_2 \calA\sinc(t_2 \Omega) u^0 \\ * \end{bmatrix}
\end{align*}
for all $u^0 \in \calV$. Analogously to~\ref{item:properties:cos_1}, this equation holds for every $u^0 \in \calH$ as well.

\ref{item:properties:cos_3} Since~\eqref{eq:PDE:firstOrdNoDamp} is invariant under time reversion, it holds that $t\calA\sinc(t\Omega)$ is odd. Therefore, it holds that 
\begin{equation*}
	\begin{bmatrix} u^0\\ 0 \end{bmatrix} 
	= (e^{t\calZ} e^{-t\calZ}) 
	\begin{bmatrix} u^0\\ 0 \end{bmatrix} 
	= \begin{bmatrix} \cos^2(t\Omega) u^0 + t\sinc(t\Omega)\, t\calA\sinc(t\Omega) u^0 \\ 0 \end{bmatrix},
\end{equation*}
which completes the proof. 
\end{proof}
Before we investigate the convergence order we state the following auxiliary lemma. 
\begin{lemma}\label{lem:Xn}
Let $t>0$ and $\cos(t\Omega)$ be defined as in Lemma~\ref{lem:properties:cos}. Suppose that $X_0 \coloneqq \id$, $X_1\coloneqq 2 \cos(t \Omega)$, and 
\[
	X_n \coloneqq 2 \cos(t \Omega)X_{n-1}-X_{n-2}\qquad
	\text{for }	n\ge 2.
\]
For $n\ge 0$ it then holds that
\begin{align}
\label{eq:formulaXn}
	X_n 
	= \begin{cases} \id + 2 \sum_{k=1}^{n/2} \cos(2k t \Omega), 
	 & \text{if $n$ is even}\\
	 2 \sum_{k=1}^{(n+1)/2} \cos\big((2k -1)t \Omega\big), 
	 & \text{if $n$ is odd}
	\end{cases}.
\end{align}
\end{lemma}
\begin{proof}
The proof follows by mathematical induction. The assertion is clear for $X_0$ and $X_1$. If~$n$ is even, then we have  
\begin{align*}
	X_n 
	&\overset{\phantom{\text{Lem}.~\ref{lem:properties:cos} a)}}{=} 2 \cos(t \Omega) X_{n-1} - X_{n-2}\\
	&\overset{\phantom{\text{Lem}.~\ref{lem:properties:cos} a)}}{=} 4 \cos(t \Omega) \sum_{k=1}^{n/2} \cos\big((2k -1)\,t\Omega\big) - \id - 2 \sum_{k=1}^{n/2-1} \cos(2k t \Omega)\\
	&\overset{\text{Lem}.~\ref{lem:properties:cos}\,\ref{item:properties:cos_2}}{=}  2\,\sum_{k=2}^{n/2} \Big[ \cos(t \Omega) \cos\big((2k -1)\,t\Omega\big) -  
	\sin\big((2k-1)\,t\Omega\big)\sin(t\Omega) \Big] + 4 \cos^2(t\Omega) - \id\\
	&\overset{\text{Lem}.~\ref{lem:properties:cos}\,\ref{item:properties:cos_2}}{=}  2\,\sum_{k=1}^{n/2} \cos(2kt \Omega) + 2 \cos^2(t\Omega) + 2 \sin^2(t\Omega) - \id\\
	&\overset{\text{Lem}.~\ref{lem:properties:cos}\,\ref{item:properties:cos_3}}{=}  2\,\sum_{k=1}^{n/2} \cos(2kt \Omega) + \id.
\end{align*}
Note that we use $2kt = (2k+1)t + (-t)$ within the first application of Lemma~\ref{lem:properties:cos}\,\ref{item:properties:cos_2} in combination with an index shift. 
Moreover, we would like to emphasize at this point that $\sin(t\Omega)$ may not be defined. However, due to $x\sinc(x) = \sin(x)$, the above expressions only contain $\Omega^2$, which makes it well-defined. 
If, on the other hand, $n$ is odd, then we get 
\begin{align*}
	X_n 
	&\overset{\phantom{\text{Lem}.~\ref{lem:properties:cos} a)}}{=} 2 \cos(t \Omega) X_{n-1} - X_{n-2}\\
	&\overset{\phantom{\text{Lem}.~\ref{lem:properties:cos} a)}}{=} 2 \cos(t \Omega) \Big( \id + 2 \sum_{k=1}^{(n-1)/2} \cos(2kt \Omega)\Big) - 2 \sum_{k=1}^{(n-1)/2} \cos\big((2k -1)\,t\Omega\big)\\
	&\overset{\text{Lem}.~\ref{lem:properties:cos}\,\ref{item:properties:cos_2}}{=}  2 \cos(t \Omega) + 2 \sum_{k=1}^{(n-1)/2} \Big[ \cos(t \Omega)\cos(2k t\Omega)-\sin(t \Omega)\sin(2k t\Omega) \Big]\\
	&\overset{\text{Lem}.~\ref{lem:properties:cos}\,\ref{item:properties:cos_2}}{=}  2 \cos(t \Omega) + 2\sum_{k=1}^{(n-1)/2} \cos\big((2k+1)\,t\Omega\big)\\
	&\overset{\phantom{\text{Lem}.~\ref{lem:properties:cos} a)}}{=} 2\sum_{k=1}^{(n+1)/2} \cos\big((2k-1)\,t\Omega\big).
	\qedhere
\end{align*}
\end{proof}
We now turn to the second main result of this paper, namely the proof of second-order convergence of the proposed two-step Gautschi-type integrator~\eqref{eq:Gautschi:PDAE}. 
\begin{theorem}[Second-order convergence, Gautschi-type integrator]
\label{th:Gautschi:convergence}
Let Assumptions~\ref{assB}, \ref{assA}, and~\ref{assRHS}~(a)
be satisfied and assume consistent initial data $u^0 \in \calV$, $w^0 \in \calH$ leading to $u^1$ satisfying~$\|u^1-u(\tn{1})\|_{\calH} \le \tau^3 E_\text{init}$. Further, assume right-hand sides 
\[
	f(\,\cdot\,, u(\cdot)) \in C^2([0,T], \cH), \qquad 
	g\in C^4([0,T], \Q^*).
\]
Then the Gautschi-type integrator~\eqref{eq:Gautschi:PDAE} is second-order accurate with the error estimate  
\[
	\|u^n - u(\tn{n})\|_{\cH}
	\le \tau^2\, \tn{n}\, e^{\kappa \tn{n}} \cosh(\tn{n} L^{\sfrac 12})\,  \big( E_\text{init} + \tn{n} C_{f,g} \big)
	+ \calO(\tau^3)
\]
with
\[
	C_{f,g} 
	\coloneqq \big\| \ddot f(\,\cdot\,,u)\,\big\|_{L^\infty(0,\tn{n};\calH^*)} 
	+ C_{\V\to\cH}C_{\calB^-} \big\| g^{(4)} \big\|_{L^\infty(0,\tn{n};\calQ^*)}
\]
including the embedding constant~$C_{\V\to\cH}$, the continuity constant~$C_{\calB^-}$ of the operator~$\calB^-$, $\kappa = C_2 / \sqrt{\mu_1}$ with the continuity constant of~$\calA_2$ and the ellipticity constant of~$\calA_1$, and $L \coloneqq \exp(\kappa\, \tn{n}) L_{\cH}$. 
\end{theorem}
\begin{proof}
As in the beginning of Section~\ref{sec:Gautschi:scheme}, we consider~\eqref{eq:slnFormulaPDAE} at times $\tn{n+1}$ and $\tn{n-1}$. In combination with the approximate solution and formula~\eqref{eq:Gautschi:PDAE}, we obtain for the error~$\errunGau{n} \coloneqq  u^n - u(\tn{n})$ the equation 
\begin{align*}
	\errunGau{n+1} + \errunGau{n-1}
	%
	&= 2 \cos(\tau\Omo) \errunGau{n} + 2\, \big(1 - \cos(\tau\Omo)\big) \calAo^{-1}\, \big[ f(\tn{n}, u^n) - \calB^-\ddot g^n \big] \\
	&\qquad - \int_0^\tau (\tau-s)\, \sinc((\tau-s)\, \Omo)\, \big[ f(\tn{n}+s, u(\tn{n}+s)) - \calB^-\ddot g(\tn{n}+s) \\ 
	&\hspace{5.8cm}+ f(\tn{n}-s, u(\tn{n}-s)) - \calB^-\ddot g(\tn{n}-s) \big] \ds.
\end{align*}
Using the equality 
\[
	\big(1 - \cos(\tau\Omo) \big) \calAo^{-1}\, h_n
	= \int_0^\tau\, (\tau-s)\, \sinc\big((\tau-s)\, \Omo \big)\, h_n \ds
\]
with $h_n = f(\tn{n}, u(\tn{n})) - \calB^-\ddot g^n$, we get 
\begin{align}
	\errunGau{n+1} + \errunGau{n-1} 
	&= 2 \cos(\tau\Omo) \errunGau{n} + 2\, \big(1 - \cos(\tau\Omo)\big) \calAo^{-1}\, \big[ f(\tn{n}, u^n) - f(\tn{n}, u(\tn{n})) \big] \notag \\
	&\quad - \int_0^\tau (\tau-s)\, \sinc((\tau-s)\, \Omo)\, \big[ f(\tn{n}+s, u(\tn{n}+s)) - \calB^-\ddot g(\tn{n}+s) \label{eq:inproof:errorGautschi} \\ 
	&\hspace{1.3cm}- 2f(\tn{n}, u(\tn{n})) + 2\calB^-\ddot g(\tn{n}) + f(\tn{n}-s, u(\tn{n}-s)) - \calB^-\ddot g(\tn{n}-s) \big] \ds. \notag
\end{align}
Note that the second summand of the right-hand side equals the solution of~\eqref{eq:uker} with constant right-hand side $2f(\tn{n},u^n) - 2f(\tn{n},u(\tn{n}))\in \calH^*$ and homogeneous initial values (and $g=0$) evaluated at $t=\tau$. We denote this solution by $\widetilde{u}_n$ and the stability estimate from Lemma~\ref{lem:existenceLinUnconstrained} implies 
\begin{align*}
	\big\|\dot{\widetilde{u}}_n(t)\big\|_{\calH} 
	\le 2\, e^{\frac{\kappa t}{2}}\,
	\int_0^t e^{-\kappa s}\, \|f(\tn{n},u^n) - f(\tn{n},u(\tn{n}))\|_{\calH^*} \ds. 
\end{align*}
This turns into 
\begin{align*}
	\|\widetilde{u}_n(\tau)\|_{\calH} 
	\le \int_0^\tau \|\dot{\widetilde{u}}_n(t)\|_{\calH} \dt
	&\le \int_0^\tau 2\, e^{\frac{\kappa t}{2}}\, \int_0^t e^{-\kappa s}\, \|f(\tn{n},u^n) - f(\tn{n},u(\tn{n}))\|_{\calH^*} \ds \dt \\
	&\le 2\, L_{\cH}\, \|e_n\|_\calH\, \int_0^\tau e^{\frac{\kappa t}{2}}\, \int_0^t e^{-\kappa s}\, \ds \dt \\
	&\le \tau^2\, L_{\cH}\, e^{\frac{\kappa \tau}{2}}\, \|e_n\|_\calH
\end{align*}
with $L_{\cH}$ denoting the Lipschitz constant of $f$ from Assumption~\ref{assRHS}. Also, the integral term in~\eqref{eq:inproof:errorGautschi} equals the solution of~\eqref{eq:uker} evaluated at $t=\tau$ for homogeneous initial data but now with right-hand side
\[
	f(\tn{n} +t,u(\tn{n}+t)) - 2 f(\tn{n},u(\tn{n})) + f(\tn{n} -t,u(\tn{n}-t))
	- \calB^-\ddot g(\tn{n}+t) + 2 \calB^-\ddot g(\tn{n}) - \calB^-\ddot g(\tn{n}-t).
\]
We denote this solution by $\widehat{u}_n$. Using a Taylor expansion, we obtain similarly as before 
\begin{align*}
	\|\widehat{u}_n(\tau)\|_{\calH} 
	&\le \int_0^\tau \|\dot{\widehat{u}}_n(t)\|_{\calH} \dt \\
	&\le e^{\frac{\tau \kappa}{2}} \int_0^\tau \int_0^t \big\|f(\tn{n} +s,u(\tn{n}+s)) - 2 f(\tn{n},u(\tn{n})) +f(\tn{n} -s,u(\tn{n}-s)) \big\|_{\calH^*} \\ 
	&\hspace{2.8cm} + \big\|\calB^-\ddot g(\tn{n}+s) - 2 \calB^-\ddot g(\tn{n}) + \calB^-\ddot g(\tn{n}-s) \big\|_{\calH^*} \ds \dt\\
	&\le e^{\frac{\tau \kappa}{2}} \int_0^\tau \int_0^t \| s^2 \ddot f(\tn{n},u(\tn{n})) \|_{\calH^*} + \| s^2 \calB^- g^{(4)}(\tn{n}) \|_{\calH^*} + \calO(s^3) \ds \dt \\
	&\le \frac{\tau^4}{12}\, e^{\frac{\tau \kappa}{2}} \Big( \big\|\ddot f(\tn{n},u(\tn{n})) \big\|_{\calH^*} + C_{\V\to\cH}C_{\calB^-} \| g^{(4)}(\tn{n}) \|_{\Q^*} \Big) + \calO(\tau^5).
\end{align*}
To summarize, we have the recursion  
\[
	\errunGau{n+1} + \errunGau{n-1} 
	= 2 \cos(\tau\Omo) \errunGau{n} + \widetilde{u}_n(\tau) + \widehat{u}_n(\tau).
\]
By mathematical induction, one can show that this can be written in the explicit form 
\[
	\errunGau{n} 
	= -X_{n-2}\, e_0 +  X_{n-1}\, e_1 + \sum_{j=1}^{n-1} X_{n-1-j}\, \big(\widetilde{u}_j(\tau) - \widehat{u}_j(\tau) \big)
\]
with the~$X_j$ from Lemma~\ref{lem:Xn}. With formula~\eqref{eq:formulaXn} and the boundedness of~$\cos(\tau\Omo)$ shown in Lemma~\ref{lem:properties:cos}, these $X_j$ satisfy  
\[
	\|X_j\|_*
	\coloneqq \|X_j\|_{\calL(\cH,\cH)}
	\le (j+1)\, e^{2\frac{\kappa j \tau}{2}}
	= (j+1)\, e^{\kappa \tn{j}}. 
\]
and accordingly $\|X_j\|_* \le j+1$ in the symmetric case. With $C_{f,g}$ as defined above, this then leads to 
\begin{align*}
	\|\errunGau{n}\|_\cH 
	&\le 0 + \|X_{n-1}\|_* \tau^3 E_\text{init} + \sum_{j=1}^{n-1} \|X_{n-1-j}\|_* \big( \|\widetilde{u}_j(\tau)\|_\cH + \|\widehat{u}_j(\tau)\|_\cH \big) \\
	&\le n\, \tau^3\, e^{\kappa \tn{n}} E_\text{init} + \sum_{j=1}^{n-1} (n-j)\, e^{\kappa \tn{n-j}} \big( \|\widetilde{u}_j(\tau)\|_\cH + \|\widehat{u}_j(\tau)\|_\cH \big) \\
%
%
	&\lesssim \tau^2\, \tn{n}\, e^{\kappa \tn{n}} \big( E_\text{init} + \tn{n}C_{f,g} \big) 
	+ \tau^2\, e^{\kappa \tn{n}} L_{\cH}\, 
	\sum_{j=1}^{n-1} (n-j)\, \|e_j\|_\calH 
	+ \calO(\tau^3).
\end{align*}
An application of~\cite[Lem.~2]{GarSS98} finally gives the stated result.
%
\end{proof}
%
%
\subsection{Practical Implementation}
\label{sec:Gautschi:practical}
Before presenting numerical examples, we need to discuss how the introduced Gautschi-type integrator~\eqref{eq:Gautschi:PDAE} can be implemented. For this, we assume a spatial discretization leading to the matrices $A$, $B$, and $M$; cf.~Remark~\ref{rem:saddlepointproblems}.

In the following, we explain how to compute the action of $\calB^-$ and~$\cos(\tau\Omo)$ in the discrete setting. For this, $\Omega_{\ker}$ needs to be replaced by the square root of some matrix~$A_{\ker}$ (analogously to~$\calA_{\ker}$). Here, we consider $A_{\ker}$ as a mapping~$\ker B \subseteq \R^n \to \ker B \subseteq \R^n$. Since we do not want to compute the kernel of~$B$ explicitly, we will only describe the action of~$A_{\ker}$ applied to a vector in terms of an auxiliary saddle point problem rather than assembling the operator itself. Moreover, we are concerned with the efficient evaluation of the cosine-term, for which we introduce an Arnoldi algorithm. 

First, we consider the computation of $B^-g^n$ for some~$g^n\in \R^m$. Following the construction in Section~\ref{sec:PDAE:prelim}, $B^-$ should map into (the discrete version of) $\Vc$, i.e., the result should be $A$-orthogonal to the kernel of~$B$. To be precise, $B^-g^n$ should satisfy~$w^T A (B^-g^n) = 0$ for all $w\in\ker B \subseteq \R^n$. To achieve this, we set~$x = B^-g^n \in \R^n$ as the solution of the saddle point problem 
\begin{subequations}
\label{eq:Binvg}	
\begin{alignat}{5}
	A x&\ +\ & B^T \mu\, 
	&= 0, \label{eq:Binvg:a}\\
	B x & & 
	&= g^n. \label{eq:Binvg:b}	
\end{alignat}
\end{subequations}
Note that this system has a unique solution pair $(x,\mu) \in \R^n \times \R^m$, since $B$ is of full rank and $A$ is assumed to be positive definite on~$\ker B$. Obviously, equation~\eqref{eq:Binvg:b} guarantees that $Bx = BB^-g^n = g^n$. Moreover, the desired orthogonality condition is satisfied, since~$w\in\ker B$ yields due to~\eqref{eq:Binvg:a}, 
\[
	w^TAx
	= w^T (Ax + B^T\mu)
	= 0.
\]

Second, we discuss the application of the operator~$\calA_{\ker}$ in the discrete setting. Given a vector~$v\in \ker B \subseteq \R^n$, we obtain~$x = A_{\ker}v$ as the solution of the saddle point problem 
\begin{subequations}
\label{eq:Akerv}	
\begin{alignat}{5}
	M x&\ +\ & B^T \mu\, 
	&= Av, \label{eq:Akerv:a}\\
	B x & & 
	&= 0 \label{eq:Akerv:b}.	
\end{alignat}
\end{subequations}
Due to the assumed properties on $M$ and $B$, this system is again uniquely solvable. In order to see that~$x$ is indeed the desired vector, first note that equation~\eqref{eq:Akerv:b} ensures that~$x$ is an element of~$\ker B$. The first equation, on the other hand, corresponds to 
\[
	\ip{x,w}_\cH 
	= \langle\calA v, w\rangle
	= \langle\calA_{\ker} v, w\rangle, \qquad
	w\in\Vo 
\]
in the continuous setting. The application of the operator~$\calA_{\ker}$ is crucial for the computation of the cosine-term. 
%
For the intermediate step of computing $b^n$, we also need the inverse of $\calA_{\ker}$. In the discrete setting, $b^n = A_{\ker}^{-1}v$ is given by the solution of 
\begin{alignat*}{5}
	A b^n&\ +\ & B^T \mu\, 
	&= v, \\ 
	B b^n & & 
	&= 0.	
\end{alignat*}

Finally, we are concerned with the computation of $\cos(\tau \Omega_\text{ker})x_0$ for a given vector~$x_0\in \ker B \subseteq \R^n$. Since the Taylor series of $\cos(\tau \Omega_\text{ker})$ only contains even powers of~$\Omega_\text{ker}$, we do not need to compute~$\Omega_\text{ker}$ itself. It suffices to compute~$A_\text{ker}$ or rather its action on a vector. Hence, we consider the Krylov subspace spanned by the vectors~$A_\text{ker}^k x_0$ for $k\ge0$. 
Following the general idea of Krylov subspace methods (see, e.g., \cite{Saa92,EieE06}), we introduce 
\begin{equation*}
	\calK_r 
	\coloneqq \calK_r(A_{\ker},x_0) 
	\coloneqq \operatorname{span} \big\{ x_0, A_{\ker} x_0,\ldots, A_{\ker}^{r-1}x_0 \big\}. 
\end{equation*}
Obviously, this space does not rely on the computation of $\Omega_{\ker}$. Moreover, we would like to recall that~$A_\text{ker}$ is never computed explicitly. Instead, only the action of $A_{\ker}$ is computed via the saddle point problem~\eqref{eq:Akerv}. In a second step, we generate an orthogonal basis of~$\calK_r$ using the Arnoldi algorithm with $v_1 = {x_0}/{\|x_0\|}$ as initial vector. Collecting these vectors in the isometric matrix $V_r \in \R^{n\times r}$, we obtain~$V_r^T A V_r = H_r$ with an upper Hessenberg matrix $H_r\in \R^{r\times r}$. This then yields the approximation
\begin{equation*}
	\cos(\tau \Omega_\text{ker})x_0
	\approx \|x_0\|\, V_r \cos(\tau H_r) e_1
\end{equation*}
with unit basis vector $e_1 \in \R^r$. For the computation of~$\cos(\tau H_r)$, which is of moderate size, one may use methods described in~\cite[Ch.~12]{Hig08}. For our numerical examples, we use the algorithm introduced in~\cite{AlMHR15}. 
\begin{remark}
Using the Arnoldi method for the computation of the matrix cosine leaves open the question of stopping criteria. 
Here, one may apply the determination criterion presented in~\cite{HocPSTW15}. 
In the numerical experiment of Section~\ref{sec:numerics}, $r=3$ was already sufficient. 
\end{remark}
%
%
%
\section{Numerical Example}
\label{sec:numerics} 
%
%
%
%
%
%
As numerical example, we consider the semi-linear wave equation with kinetic boundary conditions without damping. In the strong form, the system equations read 
\begin{subequations}
\label{eq:kineticBC}
\begin{align}
	\ddot u - \Delta u + u 
	&= f_\Omega(t,u)\qquad \text{in }\Omega, \label{eq:kineticBC:a}\\
	\ddot u - \Delta_\Gamma u + \partial_n u 
	&= f_\Gamma(t,u)\qquad \text{on }\Gamma \coloneqq \partial\Omega. \label{eq:kineticBC:b}
\end{align}
\end{subequations}
Here, $\Delta_\Gamma$ denotes the Laplace--Beltrami operator, i.e., the Laplacian along the boundary; cf.~\cite[Ch.~16.1]{GilT01}. For the inhomogeneities, we set
\[
	f_\Omega(t,u) = \sin(t),\qquad
	f_\Gamma(t,u) = - u^3 + u.
\]
Note that $f_\Omega$ satisfies part (a) and (b) of Assumption~\ref{assRHS}, but $f_\Gamma$ only satisfies part (b).

To show well-posedness of~\eqref{eq:kineticBC}, the system can be written in an abstract form; see~\cite{HipK20,HocL20}. 
To obtain the PDAE structure~\eqref{eq:PDAE}, we consider an alternative formulation as introduced in~\cite{Alt19,Alt23}. For this, we introduce a second variable~$p\coloneqq u|_\Gamma$, which is only defined on the boundary. 
With this, we can interpret~\eqref{eq:kineticBC:a} as a dynamic equation for $u$ and~\eqref{eq:kineticBC:b} as a dynamic equation for $p$, both coupled through the mentioned connection. We further define
\[
	\V \coloneqq H^1(\Omega)\times H^1(\Gamma), \qquad
	\cH \coloneqq L^2(\Omega)\times L^2(\Gamma), \qquad
	\Q \coloneqq H^{-1/2}(\Gamma).
\]
As shown in~\cite{Alt23}, the weak formulation of~\eqref{eq:kineticBC} can be written as 
\begin{subequations}
\label{eq:PDAE:kinetic}
	\begin{align}
	\begin{bmatrix} \ddot u \\ \ddot p  \end{bmatrix}
	+ \begin{bmatrix} \calA_\Omega &  \\  & \calA_\Gamma \end{bmatrix}
	\begin{bmatrix} u \\ p  \end{bmatrix}
	+ \calB^*\lambda 
	&= \begin{bmatrix} f_\Omega(t,u) \\ f_\Gamma(t,p) \end{bmatrix} \qquad \text{in } \V^*, \\
	\calB\, \begin{bmatrix} u \\ p  \end{bmatrix} \phantom{i + \calB \lambda} &= \phantom{[]} 0\hspace{5.6em} \text{in } \Q^*, 
	\end{align}
\end{subequations}
which has exactly the structure of~\eqref{eq:PDAE}. Therein, the constraint operator is given by~$\calB\colon \V\to\Q^*=H^{1/2}(\Gamma)$, $\calB(u,p) \coloneqq p - \trace u$ with the usual trace operator. The kernel of $\calB$ in $\V$ is given by
\[
	\Vo 
	= \ker\calB
	= \big\{ (u,p)\in\V\ |\ \trace u = p \big\}, 
\]
leading to $\cHo = \cH$.
Thus, we are in the modified setting of Remark~\ref{rem:H_eq_H_ker}. The differential operators are given by~$\calA_\Omega\colon H^1(\Omega) \to [H^1(\Omega)]^*$ and~$\calA_\Gamma\colon H^1(\Gamma) \to [H^1(\Gamma)]^*$, 
\begin{align*}
	\langle \calA_\Omega u, v\rangle
	= \int_\Omega \nabla u \cdot \nabla v + uv \dx, \qquad
	\langle \calA_\Gamma p, q\rangle
	= \int_\Gamma \nabla_\Gamma p \cdot \nabla_\Gamma q \dx.
\end{align*}
Note that, in contrast to~\eqref{eq:kineticBC}, system~\eqref{eq:PDAE:kinetic} contains two additional variables, namely~$p$ for the trace values and a Lagrange multiplier~$\lambda\colon [0,T] \to \Q$, which enforces the constraint.  

The numerical tests are carried out on the unit disc, i.e., we consider
\[
	\Omega 
	= \big\{ (x,y)\in\R^2\ |\ x^2 + y^2 \le 1 \big\}.
\]
The spatial meshes are generated by DistMesh \cite{PerS04}, and we use piecewise linear finite element approximations. Moreover, we set $T = 1$ and, as initial condition, 
\[
	u(0) 
	= u^0 
	\coloneqq \exp\big( -20\,((x-1)^2 + y^2) \big), \qquad
	\dot u(0) 
	= \dot u^0 
	\coloneqq 0.
\]
Accordingly, $p(0)$ and $\dot p(0)$ are defined consistently. 
Within the experiment, we have used six spatial level refinements, leading to a mesh size of around $h \approx 0.0672$. The reference solution is computed on the same spatial grid by the Gautschi-type integrator with Krylov subspace dimension $r = 10$ and a time step size of $\tau_\text{ref} = 2^{-13}$. 

The obtained convergence history is presented in Figure~\ref{fig:convL2}. The displayed lines show the errors in $u$ measured in the $L^\infty(0,T;L^2(\Omega))$-norm. One can see the first-order convergence of the IMEX Euler scheme, i.e., of the implicit Euler scheme where the nonlinearity is treated explicitly, i.e., 
\begin{align*}
	- \tau\, w^{n+1} \hspace{0.52cm} + u^{n+1}\hspace{1.91cm}
	&= u^n, 
	\\
	w^{n+1} + \tau \calA u^{n+1} + \tau \calB^* \lambda^{n+1}
	&= w^n + \tau f^n,  
	\\ 
	\calB u^{n+1} \hspace{1.91cm}
	&= g^{n+1}.
\end{align*}	
We further observe the proven second-order accuracy of the IMEX scheme~\eqref{eq:IMEX_damp}. For the Gautschi-type integrator~\eqref{eq:Gautschi:PDAE}, we present three versions that differ by the dimension of the Krylov subspaces used for the computation of the matrix cosine. For dimension $r=1$, the error stagnates but already for $r=2$ the errors show second-order convergence. The results improve once more for $r=3$ but remain unchanged for larger dimensions $r>3$. We note that the choice of $f_\Gamma$ is covered by the assumption for the IMEX scheme, but not the results for the Gautschi method. However, different choices lead to similar convergence plots. 
Finally, Figure~\ref{fig:convL2} also shows the results for the fully explicit leapfrog scheme. Here, we observe the expected second-order convergence as soon as the scheme gets stable for sufficiently small time step sizes. 

If one considers the $L^\infty(0,T;H^1(\Omega))$-norm, in which the IMEX \CN error was estimated in Theorem~\ref{th:IMEX:convergence}, we observe the same orders but with larger error constants.  
\begin{figure}
%
%
\begin{tikzpicture}

\begin{axis}[%
width=4.5in,
height=2.2in,
scale only axis,
xmode=log,
xmin=0.0002,
xmax=0.3,
xminorticks=true,
xlabel={time step size $\tau$},
ymode=log,
ymin=1.2e-08,
ymax=4,
yminorticks=true,
ylabel={error in $u$}, 
axis background/.style={fill=white},
title style={font=\bfseries},
legend style={at={(0.97,0.03)}, anchor=south east, legend cell align=left, align=left, draw=white!15!black}
]

\addplot [color=mycolor1, line width=1.5pt, mark=*, mark size=2.5]
 table[row sep=crcr]{%
	0.25	0.13231349\\
	0.125	0.10047058\\
	0.0625	0.07253974\\
	0.03125	0.048752931\\
	0.015625	0.030677557\\
	0.0078125	0.018240898\\
	0.00390625	0.010313642\\
	0.001953125	0.005594367\\
	0.0009765625	0.0029454985\\
	0.00048828125	0.001521673\\
	0.000244140625	0.00077675614\\
};
\addlegendentry{IMEX Euler} 

\addplot [color=mycolor2, line width=1.5pt, mark=triangle*, mark size=3.0]
table[row sep=crcr]{%
	0.25	0.095147238\\
	0.125	0.04740909\\
	0.0625	0.02049387\\
	0.03125	0.007037195\\
	0.015625	0.0021214459\\
	0.0078125	0.0006682096\\
	0.00390625	0.00019542444\\
	0.001953125	4.9751331e-05\\
	0.0009765625	1.2461085e-05\\
	0.00048828125	3.1161527e-06\\
	0.000244140625	7.7904691e-07\\
};
\addlegendentry{IMEX CN} 

\addplot [color=mycolor3!25!white, line width=1.5pt, mark=square*, mark size=2.8]
table[row sep=crcr]{%
	0.125	2.9101839\\
	0.0625	2.9636563\\
	0.03125	2.017452\\
	0.015625	1.2190186\\
	0.0078125	0.88987993\\
	0.00390625	0.75769521\\
	0.001953125	0.70051781\\
	0.0009765625	0.67428643\\
	0.00048828125	0.66177465\\
	0.000244140625	0.65567137\\
};
\addlegendentry{Gautschi ($r=1$)} 

\addplot [color=mycolor3!40!white, line width=1.5pt, mark=square, mark size=2.8]
table[row sep=crcr]{%
	0.25	2.4557014\\
	0.125	1.3502588\\
	0.0625	0.48398435\\
	0.03125	0.23941442\\
	0.015625	0.0099419191\\
	0.0078125	0.0024957581\\
	0.00390625	0.00062673088\\
	0.001953125	0.00015693559\\
	0.0009765625	3.926216e-05\\
	0.00048828125	9.8198502e-06\\
	0.000244140625	2.4564866e-06\\
};
\addlegendentry{Gautschi ($r=2$)} 

\addplot [color=mycolor3, line width=1.5pt, mark=square*, mark size=2.8]
table[row sep=crcr]{%
	0.25	3.421906\\
	0.125	1.0177151\\
	0.0625	0.23353575\\
	0.03125	0.0026069048\\
	0.015625	0.00011594449\\
	0.0078125	6.7425183e-06\\
	0.00390625	1.1259311e-06\\
	0.001953125	3.1497388e-07\\
	0.0009765625	8.0469759e-08\\
	0.00048828125	1.8819824e-08\\
	0.000244140625	3.0911977e-09\\
};
\addlegendentry{Gautschi ($r=3$)} 


\addplot [color=mycolor5, line width=1.5pt, mark=diamond*, mark size=2.]
table[row sep=crcr]{%
0.015625	0.0011409142\\
0.0078125	0.00037726543\\
0.00390625	9.8606651e-05\\
0.001953125	2.4792172e-05\\
0.0009765625	6.2120785e-06\\
0.00048828125	1.5550097e-06\\
0.000244140625	3.8907454e-07\\
};
\addlegendentry{leap-frog} 


\addplot [color=gray, dashed, line width=1.25pt, forget plot]
table[row sep=crcr]{%
	0.25	0.25\\
	0.00048828125	0.00048828125\\
	0.000244140625	0.000244140625\\
};

\addplot [color=gray, dotted, line width=1.25pt, forget plot]
  table[row sep=crcr]{%
	0.25	0.0625\\
	0.00048828125	2.38418579101562e-07\\
	0.000244140625	5.8276953e-08\\
};

\end{axis}

\end{tikzpicture}%
	\caption{Convergence history for different time stepping schemes. Plot shows error in $u$, measured in the $L^\infty(0,T;L^2(\Omega))$-norm. The gray lines indicate orders $1$ (dashed) and $2$ (dotted). }	
	\label{fig:convL2}
\end{figure}
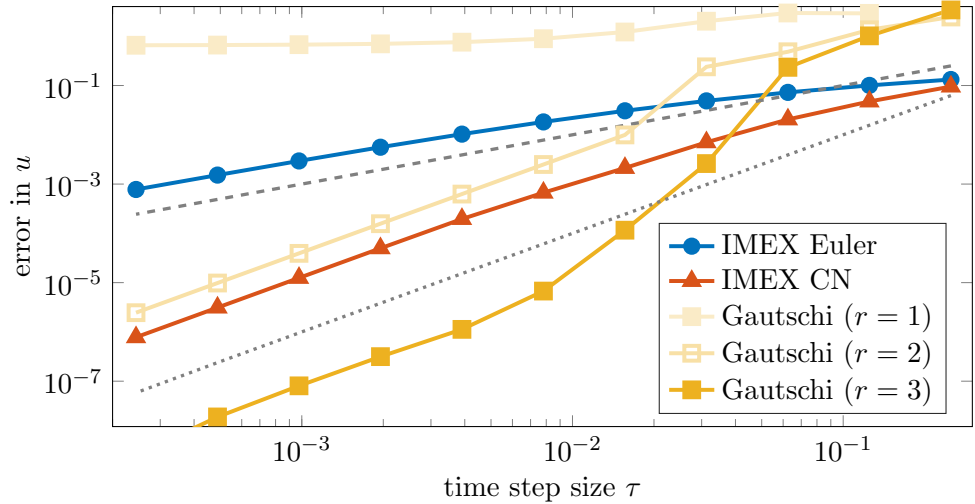

The code is available at \texttt{https://github.com/altmannr/codeForGautschiCNPaper}.
%
%
\section{Conclusion}
\label{sec:conclusion}
In this paper, we have introduced and analyzed two novel time integration schemes for semi-linear wave-type systems restricted by a linear constraint. More precisely, we have considered an implicit--explicit version of the \CN scheme as well as an integrator of Gautschi-type. For the second-order convergence proofs, we combine well-known results for the dynamical part with certain saddle point problems in order to include the constraints. 

We emphasize that the error analysis of the IMEX method can be easily adapted to the (original) \CN scheme. This, however, does not seem computationally attractive due to the nonlinear saddle problems that have to be solved in each time step.
In addition, a combination of the two analyses seems to be a good starting point to derive error bounds also for more general exponential integrators applied to constrained systems.
%
%
\section*{Acknowledgments}  
R.~Altmann acknowledges funding by the Deutsche Forschungsgemeinschaft (DFG, German Research Foundation) through the project 446856041.
B.~D\"orich acknowledges funding by the Deutsche Forschungsgemeinschaft (DFG, German Research Foundation) through the Project-ID 258734477 -- SFB 1173. 
Finally, C.~Zimmer acknowledges support by the European Research Council (ERC) under the European Union's Horizon 2020 research and innovation programme (Grant agreement No.~865751 --  RandomMultiScales).
%
%

\bibliographystyle{alpha} 
\bibliography{bib_Gautschi}
\appendix

\end{document}